\documentclass[english]{amsart}
\usepackage[english]{babel}

\usepackage{hyperref}
\usepackage{verbatim}
\usepackage{amsfonts,amssymb,mathtools}
\usepackage{amsmath}
\usepackage{bbm}
\usepackage{hyperref}

\usepackage{graphicx,xcolor}

\definecolor{green}{RGB}{89,169,58}
\definecolor{red}{RGB}{224,61,42}
\definecolor{blue}{RGB}{63,78,181}

\usepackage{subcaption}
\usepackage{tikz}
\usetikzlibrary{arrows}

\usepackage[style=numeric,maxcitenames=2,maxbibnames=10,doi=false,isbn=false,url=false,natbib=true,giveninits=true,hyperref,bibencoding=utf8,backend=biber]{biblatex}
\AtEveryBibitem{%
  \clearfield{note}%
  \clearfield{issn}%
  \clearlist{language}%
  }

\newcommand{\N}{\ensuremath{\mathbb{N}}}
\newcommand{\R}{\ensuremath{\mathbb{R}}}

\newcommand{\Q}{\ensuremath{\mathbb{Q}}}
\newcommand{\E}{\ensuremath{\mathbb{E}}}
\renewcommand{\P}{\ensuremath{\mathbb{P}}}

\newcommand{\ind}[1]{\ensuremath{\mathbbm{1}_{\left\{#1\right\}}}}
\newcommand{\diff}{\mathop{}\mathopen{}\mathrm{d}}
\newcommand{\cal}[1]{\ensuremath{\mathcal{#1}}}
\newcommand\croc[1]{\left\langle #1\right\rangle}
\newcommand\steq[1]{\stackrel{\text{\rm #1.}}{=}}

\def\eps{\varepsilon}
\def\cadlag{c\`adl\`ag }
\def\Pois{{\rm Pois}}

\newtheorem{proposition}{Proposition}
\newtheorem{definition}[proposition]{Definition}

\newtheorem{theorem}[proposition]{Theorem}
\newtheorem*{theoremI}{Theorem}
\newtheorem{corollary}[proposition]{Corollary}

\title[Stochastic Chemical Reaction Networks]{Stochastic Chemical Reaction Networks with Discontinuous Limits and AIMD processes}
\date{\today}
\author[L. Laurence]{Lucie Laurence${ }^1$}
\email{Lucie.Laurence@inria.fr}
\address[L. Laurence, Ph. Robert]{INRIA Paris, 2 rue Simone Iff, 75589 Paris Cedex 12, France}
\author[Ph. Robert]{Philippe Robert}
\email{Philippe.Robert@inria.fr}
\urladdr{http://www-rocq.inria.fr/who/Philippe.Robert}
\thanks{${}^1$Supported by PhD grant of ENS-PSL}

\begin{document}

\begin{abstract}
In this paper we study a class of stochastic chemical reaction networks (CRNs) for which chemical species are created by a sequence of chain reactions.  We prove that under some convenient conditions on the initial state, some of these networks exhibit  a discrete-induced transitions (DIT) property: isolated, random, events have a direct impact on the macroscopic state of the process. If this phenomenon has already been noticed in several CRNs, in auto-catalytic networks in the literature of physics in particular, there are up to now few rigorous studies in this domain.   A scaling analysis of several cases of such CRNs  with several classes of  initial states is achieved.  The DIT property is investigated for the case of a CRN with four nodes. We show that on the normal timescale and for  a subset of (large) initial states and for convenient Skorohod topologies,  the scaled process converges in distribution to a  Markov process with jumps, an  Additive Increase/Multiplicative Decrease (AIMD)  process. This asymptotically discontinuous limiting behavior is a consequence of a DIT property due to random, local, blowups of jumps occurring during small time intervals. With an explicit representation of  invariant measures of AIMD processes and time-change arguments,  we show that, with a speed-up of the timescale, the scaled process is converging in distribution to a continuous deterministic function. The DIT analyzed in this paper is connected to a simple chain reaction between three chemical species and is therefore likely to be a quite generic phenomenon for a large class of CRNs.
\end{abstract}

\maketitle

 \vspace{-5mm}

\bigskip

\hrule

\vspace{-3mm}

\tableofcontents

\vspace{-1cm}

\hrule

\bigskip

\section{Introduction}
The  chemical reaction  network (CRN) with $m$ chemical species considered in this paper is represented as 
\begin{multline}\label{Si}
\emptyset \xrightharpoonup{\kappa_{0}} S_{1} \xrightharpoonup{\kappa_{1}}  S_1{+}S_{2} \xrightharpoonup{\kappa_2}\cdots\\\xrightharpoonup{\kappa_i} S_i{+}S_{i+1} \xrightharpoonup{\kappa_{i+1}} S_{i+1}{+}S_{i+2}\xrightharpoonup{\kappa_{i+2}}\cdots\\\xrightharpoonup{\kappa_{m-1}}  S_{m-1}{+}S_{m}\xrightharpoonup{\kappa_{m}} S_{m}\xrightharpoonup{\kappa_{m+1}}\emptyset.
\end{multline}

 The associated Markov process  $(X(t)){=}(X_i(t),1{\le}i{\le}m)$ lives in the state space $\N^m$.  The kinetics considered for these networks is the classical {\em law of mass action}. See~\citet{guldberg1864studies} and~\citet{Voit2015} for example. It  has the $Q$-matrix $Q{=}(q(x,y),x,y{\in}\N^m)$ defined by, for $x{\in}\N^m$ and $2{\le}i{<}m{-}1$,
\begin{equation}\label{Qmat}
\begin{cases}
  &q(x,x+e_{i+1}{-}e_{i-1}) =  \kappa_{i}x_ix_{i-1},  \qquad q(x,x+e_{1})=  \kappa_{0},\\
    &q(x,x+e_{2}) =  \kappa_{1}x_1,    \qquad q(x,x{-}e_{m-1}) = \kappa_{m}x_mx_{m-1},\\
    &q(x,x{-}e_{m}) =  \kappa_{m+1}x_m,
\end{cases}
\end{equation}
where $(\kappa_i){\in}(0,{+}\infty)^{m+2}$ is the vector of the reaction rates.

Note that the molecules of chemical species $S_1$ are created from an external input and the chemical species $S_m$ vanish independently of the other chemical species. See Section~\ref{ModelSec} for a detailed presentation of the mathematical context of these CRNs. 

We have two important features of this class of CRNs.
\begin{enumerate}
\item {\em Quadratic Rates}.\\
Due to the assumption of the law of mass action, the rate of most of reactions of Relation~\eqref{Si} is a quadratic function of the state. It is a polynomial function in general.
  
The polynomial dependence for the reaction  rates has the consequence that different timescales may coexist in the dynamical behavior of CRNs.  In our example, if all coordinates are of the order of $N$, the rate of the reaction in the middle of Relation~\eqref{Si} is $O(N^2)$, for the reaction $S_m{\rightharpoonup}\emptyset$, it is $O(N)$ and for $\emptyset{\rightharpoonup}S_1$, it is only $O(1)$. 

\item {\em Boundary Behaviors}.\\
The reaction in the middle of Relation~\eqref{Si} does not occur if either $x_i$ or $x_{i+1}$ is null. A molecule of $S_i$ may be transformed into  a molecule of $S_{i+2}$ only if there is at least  a molecule of $S_{i+1}$, even if the $(i{+}1)$th coordinates is not changed by the reaction. This is a boundary effect, some reactions do not occur on the boundary of the state space.

  These boundary effects may have a strong impact on the time evolution of CRNs. For example, for the  CRN of~\citet{LR23}, 
 \[
  \emptyset \mathrel{\mathop{\xrightleftharpoons{}}}  S_1{+}S_2, \hspace{1cm} pS_1{+}S_2 \mathrel{\mathop{\xrightleftharpoons{}}}pS_1{+}2S_2,
  \]
  it has been shown, Theorem~32,  that if $p{\ge}2$ to investigate the time evolution of the CRN starting from  large initial states of the order of $N$, one has to speed-up time by a factor $N^{p-1}$, i.e.  consider limit theorems with the fast timescale $t{\to}N^{p-1}t$, to get non-trivial time evolution. This is essentially due to the boundary condition that at least $p$ copies of $S_1$ are required in two reactions.

As we will see the boundary behaviors of the CRNs  of this paper have also an impact on its time evolution, through a phenomenon referred to as {\em discrete-induced transitions in the literature}. See Section~\ref{DITSec}. 
\end{enumerate}

\noindent
{\sc Auto-catalytic Reaction Networks.}\\
There are other classical classes of CRNs with a generic related way of transforming chemical species. 
The reactions for these networks are 
  \[
  S_i{+}S_{j} \rightharpoonup 2S_i,\qquad \emptyset \xrightleftharpoons[]{} S_i, \qquad 1{\le}i{\ne}j{\le}m.
  \]
See~\citet{Togashi_2001} for example.  As it will be seen, they seem to share some of the properties of the CRNs we are considering in this paper.

\subsection{Scaling with the Norm of the Initial State}
We will investigate this class of CRNs   via the convergence in distribution of its scaled sample paths. This may provide an interesting insight  on  the time evolution of these CRNs. Note that the deficiency of these CRNs is $1$, see Section~\ref{ModelSec}, the classical result of~\citet{Anderson2010} cannot be used here, so that even the existence of an invariant distribution is not known a priori.

Classically, a scaling approach is  mainly done via a scaling of reaction rates. This can be achieved in several ways.
\begin{enumerate}
\item The reaction rates are scaled so that all chemical reactions have a rate of the same order of magnitude in $N$, then with a convenient scaling of the space variable, it can be shown that the scaled Markov process converges in distribution to a deterministic dynamical system of the type described above.  See~\citet{Mozgunov} and Section~2.3 of~\citet{LR23}. The drawback of this scaling is that the scaled CRN does not really exhibit anymore  different timescales in $N$ since the transitions of the scaled process  are all  $O(1)$.
\item The reference~\citet{BallKurtz} considers several CRNs with given scaled of reaction rates with some parameter $N$, a multi-timescale analysis of several classes of such CRNs is achieved via proofs of averaging principles. See also~\citet{KangKurtz} and~\citet{KangKurtzPop}. 
\end{enumerate}

In the spirit of~\citet{LR23},  the scaling we consider in this paper do not change the basic dynamic of the CRN, in particular its reaction rates.  It is assumed that the initial state of the CRN is ``large'',  its norm is proportional to some scaling parameter $N$. We investigate the time evolution of the CRN, in particular how such a saturated initial state returns to some neighborhood of the origin.

If the initial states of two chemical species $S_i$, $S_{i+1}$ are both of the order of $N$, the rate of the reaction in the middle of Relation~\eqref{Si} is of the order of $N^2$ which is maximal for this class of CRNs.
Proposition~\ref{InitH0} of Section~\ref{GenSec} shows  that, for any $\eps{>}0$, the state of the CRN goes ``quickly'' to a set $S_N$ of states for which the indices of the coordinates whose value is greater than $\eps N$  are at distance at least two, the other coordinates being $o(N)$. 

When the number of nodes is odd and the initial state is such that the coordinates with an odd index are of the order of $N$ and all the others are $O(1)$, Theorem~\ref{TheoImp} shows that $(X_N(t)/N)$ is converging in distribution to the solution $(x(t))$ of an ODE converging to $0$ at infinity. The decay of the state of the CRN starting from this initial state is therefore observed on the normal time scale.  This is one of the few general results, with respect to $m$, we have been able to derive.  Nevertheless it turns out that small values of $m$ provide already non-trivial behaviors.

In Section~\ref{TN}, the case of the network with three nodes (chemical species), $m{=}3$, is considered. It is shown, see Proposition~\ref{30N0}, that for a set of initial states   the process $(X_N(t)/N)$ converges in distribution  to a continuous, but random, process. The stochastic fluctuations, represented by the martingales in the evolution equations vanish, as usual,  due to the scaling procedure. Nevertheless there remains a random, discrete, component in the limit. This is due a boundary behavior of the kinetics. This case provides an example of a CRN whose first order is not the solution of a set of deterministic ODEs. 

\subsection{A Scaling Picture with Discontinuous Stochastic Processes}\label{DITSec}
A CRN with four chemical species is investigated in Sections~\ref{FNI} and~\ref{FNII}. A class of initial states gives rise to a
 more complex behavior than what we have observed when $m{=}3$. We did not try a complete (cumbersome) classification of initial states from this point of view as it has been done for $m{=}3$, but we do believe that this is {\em the} interesting  class of initial states.

Recall that $N$ is the scaling parameter of the initial state.  The initial states considered are of the type $(0,N,0,0)$,  with the convention that "$0$", resp. ``$N$'',  means $O(1)$, resp.  $O(N)$. We show that the process lives  in  the subset of  the state space of elements of the type  $(0,N,0,\sqrt{N})$ and that the decay of the norm of the state occurs on the timescale $(\sqrt{N}t)$.  More important, this decay is in fact based on a Discrete-Induced Transitions phenomenon (DIT) which  we now describe. 
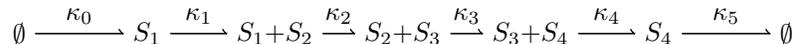
\begin{figure}[ht]
  \centerline{
\scalebox{1.0}{
    \begin{tikzpicture}[->,node distance=1.7cm]
                        \node (O) [below] {$\emptyset$};
                        \node (A) [right of=O] {$S_1$};
                        \node (B) [right of=A] {$S_1{+}S_2$};
                        \node (C) [right of=B] {$S_2{+}S_3$};
                        \node (D) [right of=C] {$S_3{+}S_4$};
                        \node (E) [right of=D] {$S_4$};
                        \node (F) [right of=E] {$\emptyset$};
                        \draw[-left to] (O)  -- node[above] {} (A);
                        \draw[-left to] (A)  -- node[above] {} (B);
                        \draw[-left to] (B)  -- node[above] {} (C);
                        \draw[-left to] (C)  -- node[above] {} (D);
                        \draw[-left to] (D)  -- node[above] {} (E);
                        \draw[-left to] (E)  -- node[above] {} (F);
                        \draw[-left to] (O)  -- node[above] {$\kappa_{0}$} (A);
                        \draw[-left to] (A)  -- node[above] {$\kappa_{1}$} (B);
                        \draw[-left to] (B)  -- node[above] {$\kappa_{2}$} (C);
                        \draw[-left to] (C)  -- node[above] {$\kappa_{3}$} (D);
                        \draw[-left to] (D)  -- node[above] {$\kappa_{4}$ } (E);
                        \draw[-left to] (E)  -- node[above] {$\kappa_{5}$ } (F);
  \end{tikzpicture}}}
 \caption{CRN with $4$ chemical species}\label{Fig2}
\end{figure}

\subsection*{Discrete-Induced Transitions (DIT)}
One of the early references to this phenomenon is~\citet{Togashi_2001} where the term has been coined apparently. In the context of auto-catalytic CRNs (see above)  with few nodes, it is observed, via numerical experiments, that  a limited number molecules of one chemical species can switch the entire bio-chemical state of a system.  It is characterized by the fact that the state variable alternates between two subsets of the state spaces. From a biological point of view,  the role of this set of molecules can be seen as a ``switch'' which can block or activate a set of  chemical reactions. See also~\citet{Togashi_2007} and~\citet{Saito}. The mathematical characterization of this phenomenon has mainly been done with the analysis of the  associated invariant measure. Intuitively the DIT phenomenon should be expressed by the fact that the invariant distribution is concentrated on at least two subsets of the state space, a bi-modal distribution. See~\citet{Bibbona}, \citet{Hoessly2019} and~\citet{GP23}. Note however that it does not give a dynamical picture of this phenomenon, such as an estimation of the sojourn time in a subset of states before visiting another subset. This is associated to the metastability  property of statistical physics. See~\citet{Bovier}. A related phenomenon has also been analyzed in~\citet{delSole}.

Returning to our CRN with four chemical species and the initial state of the type $(0,N,0,\sqrt{N}$), it turns out that the growth of the fourth coordinate $(X_4(t))$ occurs only during time intervals whose duration are $O(1/\sqrt{N})$ and during them there is a large number of positive jumps of this process, of the order of $\sqrt{N}$. Recall that this phenomenon is only due to the law of mass action which drives the kinetics of the CRN. This is where boundary effects have a significant impact. 

  The switch effect occurs during these small time intervals, the occurrence of them is driven by the isolated instants of creation of particles of chemical species $S_3$. The duration of these time intervals vanishes in the limit, so we do not have a metastability-related phenomenon as discussed in the literature. Nevertheless they play a critical role in the kinetics of the system, since this is at these instants, and only there, that the second coordinate $(X_2(t))$ can decrease.

  \subsection*{AIMD Processes}
The asymptotic behavior of $(X_4(t)/\sqrt{N})$ can be described in terms of  jump processes belonging to a class of AIMD processes on $\R_+$. These processes exhibit an exponential decay between jumps or a negative jump proportional to the current value of the state (Multiplicative Decrease) and, for the increasing part,  positive jumps depending on the current state or a linear growth (Additive Increase). There are two classes of AIMD processes in our analysis, their infinitesimal generators $\Omega_0$ and $\Omega_1$ are given by, for  bounded function $f{\in}{\cal C}_1(\R_+)$ and $v{\ge}0$,
\begin{align}
  \Omega_0(f)(v) &= -\frac{1}{\gamma}v f'(v)+\int_0^{+\infty}\left(f\left(\sqrt{v^2{+}2\beta s}\right){-}f(v)\right)e^{-s}\diff s,\label{Om0}\\
  \Omega_1(f)(v) &= \frac{1}{\gamma}f'(v) {+}v\int_0^1 \left(f\left(v u^{\beta}\right){-}f\left(v\right)\right)\diff u,\label{Om1}
\end{align}
  for some constants $\beta$, $\gamma{>}0$. 

  AIMD processes play also an important role in several classes of stochastic models, in mathematical finance, see~\citet{Bertoin} and~\citet{Yor} for an overview, communication networks~\citet{GRZ}, or genomics~\citet{Cowan}. An analogous multiplicative property has also been observed in other CRNs. See Section~8  of~\citet{LR23} for example. On the normal timescale, the main convergence result is given by the following theorem. 

\begin{theoremI}
  If  $X_N(0){=}(0,N,1,v_N)$,  with $(v_N/\sqrt{N})$ converging to $v{\ge}0$, then for the convergence in distribution for the $M_1$-Skorohod topology and also the $S$-topology,
  \[
  \lim_{N\to+\infty} \left(\frac{X_2^N(t)}{N},\frac{X_4^N(t)}{\sqrt{N}}\right)=(1,V(\kappa_0 t)),
  \]
  where $(V(t))$ is the Markov process on $\R_+$ whose infinitesimal generator $\Omega_0$ is given by, for $f{\in}{\cal C}_c^1(\R_+)$ and $x{\in}\R_+$, 
  \[
\Omega_0(f)(x) = -\frac{\kappa_5}{\kappa_0} x f'(x)+\int_0^{+\infty}\left(f\left(\sqrt{x^2{+}2\frac{\kappa_3}{\kappa_4}s}\right){-}f(x)\right)e^{-s}\diff s
\]
\end{theoremI}
The limit of the scaled process is not only random, as in Section~\ref{0N03} for $m{=}3$, but also discontinuous. 
The fact that the limit is a jump process is in fact a consequence of the large number of positive jumps during small time intervals. The $M_1$-Skorohod topology and the $S$-topology are weaker than the classical  $J_1$-Skorohod topology which does not allow accumulation of jumps. See the discussion in Section~\ref{FNI}.

The convergence result of $(X_2(t)/N)$ to  $(1)$ implies that the normal timescale is too slow to observe a decay of the norm. See~\citet{LR23} for a discussion of this phenomenon.  We prove the following convergence in distribution, it involves a speed-up of the timescale by a factor of $\sqrt{N}$. 
\begin{theoremI}
If $X(0){=}(0,N,0,v_N)$, with $(v_N/\sqrt{N})$ converging to $v{\ge}0$, then  the relation
\[
  \lim_{N\to+\infty}\left(\frac{X_2\left(\sqrt{N}t\right)}{N}, t{<}t_\infty\right)=\left(\left(1{-}\frac{t}{t_\infty}\right)^2, t{<}t_\infty\right)
  \]
holds for the convergence in distribution,  with
  \[
  t_\infty=\left.\sqrt{2}\,\frac{\kappa_5}{\kappa_0^2}\sqrt{\frac{\kappa_4}{\kappa_3}}\,\Gamma\left(\frac{\kappa_0}{2\kappa_5}\right)\right/\Gamma\left(\frac{\kappa_5{+}\kappa_0}{2\kappa_0}\right),
  \]
where $\Gamma$ is the classical {\rm Gamma} function. 
\end{theoremI}
See Theorem~\ref{AP2}. The proof of this result relies on several ingredients:
\begin{enumerate}
\item The proof of a limit result, Theorem~\ref{LimitOccApp}, related to AIMD processes associated to an infinitesimal generator of the type $\Omega_1$ of Relation~\eqref{Om1}.  An explicit expression  of their invariant distributions, Section~\ref{AIMDSec}, is used;
\item Multiple time-changes. This is an interesting example of the efficiency of stochastic calculus in such a context;
\item The proof of a stochastic averaging principle.
\end{enumerate}
We also show that the fourth coordinate is of the order of $\sqrt{N}$, with a limit result for its associated occupation measure,  and  the values of the first and third coordinates are essentially $0$ or $1$ in the limit,  from the point of view of their contribution in the evolution equations.  

\subsection{Outline of the Paper}
Section~\ref{ModelSec} introduces the class of CRNs investigated, the associated stochastic models and the kinetic equations. Section~\ref{GenSec} proves a scaling result for a CRN with an odd number of nodes and a specific class of  initial states. A scaling analysis of the CRN with three nodes is done in Section~\ref{TN}. Section~\ref{AIMDSec} introduces and investigates the AIMD processes of interest for our scaling analysis.  Sections~\ref{FNI} and~\ref{FNII} are devoted to the CRN with four nodes and a class of initial states. 

\section{Stochastic Model}\label{ModelSec}
We first  recall the formal definitions for the model of CRN of interest. 
\begin{definition}
A {\em chemical reaction network} (CRN) with $m$ {\em chemical species}, $m{\ge}1$, is defined by a triplet $({\cal S},{\cal C},{\cal R})$,
\begin{itemize}
\item ${\cal S}{=}\{1,\ldots,m\}$ is the set of species;
\item the set of {\em complexes}, is a finite subset of $\N^m$
  \[
 {\cal C}=\{0, \{e_1\},\{e_m\},\{e_i+e_{i{+}1}\},1{\le}i{\le}m{-}1\},
 \]
 where $e_i$ is the $i$th unit vector of $\N^m$. 
\item The set of {\em chemical reactions} ${\cal R}$, is  a subset of ${\cal C}^2$,
\begin{multline*}
  {\cal R}=\{(e_i{+}e_{i+1},e_{i+1}{+}e_{i+2}), 1{\le}i{<}m{-}2\}\\\cup\{(0,e_1),(e_1,e_1{+}e_2),(e_m{+}e_{m-1},e_{m}),(e_m,0)\}
\end{multline*}
\end{itemize}
\end{definition}
The notation $0$  refers to the complex associated to the null vector of $\N^m$,  $\emptyset{=}(0)$. A chemical reaction $r{=}(y_r^-,y_r^+){\in}{\cal R}$ corresponds to the change of state, for $x{=}(x_i)$ and $y_r^\pm{=}(y_{r,i}^\pm)$, 
\[
x\longrightarrow x{+}y_r^+{-}y_r^-,
\]
provided that $y_{r,i}^-{\le}x_i$ holds for  $1{\le}i{\le}m$, i.e. there are at least $y_{r,i}^-$ copies of chemical species of type $i$, for all $i{\in}{\cal S}$, otherwise the reaction cannot happen. For the CRNs considered here, we have $y_{r,i}^-{\in}\{0,1\}$ for all $r{\in}{\cal R}$ and $1{\le}i{\le}m$. 

The Markov process $(X(t))$ is clearly irreducible on $\N^m$. Its $Q$-matrix associated to the law of mass action is given by Relation~\eqref{Qmat}.  For  $2{\le}i{<}m{-}1$, the transitions of $(X(t))$ transform chemical species $S_{i-1}$ to $S_{i+1}$ via the action of $S_{i}$. As a particle system, the particles arrive at rate $\kappa_0$ at node $1$ or at  node $2$ at rate $\kappa_1x_1$ and move on the right by steps of size $2$ to finally leave the network at either node $m{-}1$ or $m$. The boundary effect is that the   transformation  of a molecule of  $S_{i-1}$ occurs at a rate proportional to the number of  molecules of $S_{i}$.  The parity of the indices of the nodes plays an important role as it can be expected. 

It is easily seen that this CRN is {\em weakly reversible} and its {\em deficiency }  is $1$, see~\citet{Feinberg} for the definitions. 

\subsection{Definitions and Notations}\label{defnot}

It is assumed that on the probability space there are $m{+}2$ independent Poisson processes on $\R_+^2$, ${\cal P}_i$, $0{\le}i{\le}m{+}1$,  with intensity $\diff s{\otimes}\diff u$ and also an independent i.i.d. family of Poisson process ${\cal N}_\sigma$, $\sigma{>}0$,  on $\R_+^2{\times}\R_+{\times}\R_+^\N$ whose intensity measure is
\[
\diff s{\otimes}\diff t{\otimes}\sigma \exp({-}\sigma a)\diff a{\otimes}Q(\diff b),
\]
where $Q$ is the distribution on $\R_+^\N$ of an  i.i.d. sequence $(E_i)$ of exponential random variables with parameter $\kappa_5$. 

The underlying filtration used throughout the paper is $({\cal F}_t)$, with, for $t{\ge}0$,
\begin{multline*}
  {\cal F}_t=\sigma\left({\cal P}_i(A{\times}[0,s)), 0{\le}i{\le}m{+}1,\right.\\\left.
    {\cal N}_\sigma(A{\times}[0,s){\times}B{\times}C),  \sigma{>}0,A, B{\in}{\cal B}(\R_+), C{\in}{\cal B}(\R_+^\N) s{\le}t\right). 
\end{multline*}
All stopping time and martingale properties implicitly refer to this filtration. See Section~B.1 of~\citet{LR23}. If $H$ is a metric space,  ${\cal C}_c(H)$ denotes the set of continuous functions on $H$ with compact support and  ${\cal C}_c^1(\R_+)$ the subset of  ${\cal C}_c(\R_+)$ of continuously differentiable functions on $\R_+$.  For $T{>}0$ ${\cal D}([0,T),\R)$ denotes the set of \cadlag functions on $[0,T)$, that is right continuous functions with left limits at all point.

For $\rho{>}0$, we will denote by $\Pois(\rho)$ the Poisson distribution on $\N$  with parameter $\rho$.     
Due to the numerous processes that have to be considered and to avoid heavy notations, we will use in the text the same notations such as $(A_N(t))$, $(B_N(t))$,  $(Z_N(t))$ for different stochastic processes, or $(M_N(t))$ for a martingale,  in different contexts, essentially in proofs of results. Similarly, several stopping times are denoted as $\tau_N$ with possibly other indices.

\subsection{Stochastic Differential Equations}
The goal of this paper is of investigating the transient properties of the sample paths of these CRNs and in particular to describe, via a functional limit theorem, how the process $(X(t))$ starting from a ``large'' initial state comes back to a neighborhood of the origin. 
It is assumed that the sequence of initial states satisfies the relation,
\begin{equation}\label{IntCond1}
\lim_{N\to+\infty}\left(\frac{X^N_i(0)}{N}\right)=\alpha=(\alpha_i){\in}\R_+^m.
\end{equation}
The scaling parameter  $N$ used   is such that the norm of the initial state is of the order of $N$.

For $N{\ge}1$, the Markov process $(X_N(t)){=}(X_i^N(t))$ starting from $X_N(0)$ with $Q$-matrix $Q$ can be represented as the solution of the SDEs, for $1{<}i{<}m$,
\begin{equation}\label{SDEm}
\begin{cases}
\diff X_1^N(t)&{=}{\cal P}_{0}\left(\left(0,\kappa_{0}\right),\diff t\right){-}{\cal P}_{2}\left(\left(0,\kappa_{2}X_{1}^N(t{-})X_{2}^N(t{-})\right),\diff t\right),\\
\diff X_2^N(t)&{=}{\cal P}_{1}\left(\left(0,\kappa_{1}X_{1}^N(t{-})\right),\diff t\right){-}{\cal P}_{3}\left(\left(0,\kappa_{3}X_{2}^N(t{-})X_{3}^N(t{-})\right),\diff t\right),\\
\diff X_i^N(t)&{=}{\cal P}_{i-1}\left(\left(0,\kappa_{i-1}X_{i-2}^N(t{-})X_{i-1}^N(t{-})\right),\diff t\right)\\
&\hspace{3cm}    {-}{\cal P}_{i+1}\left(\left(0,\kappa_{i+1}X_{i}^N(t{-})X_{i+1}^N(t{-})\right),\diff t\right),\\
\diff X_m^N(t)&{=}{\cal P}_{m-1}\left(\left(0,\kappa_{m-1}X_{m-2}^N(t{-})X_{m-1}^N(t{-})\right),\diff t\right)\\
&\hspace{3cm} {-}{\cal P}_{m+1}\left(\left(0,\kappa_{m+1}X_{m}^N(t{-})\right),\diff t\right).
\end{cases}
\end{equation}
See~\citet{Rogers2} for example. 

The scaled process is introduced as
\begin{equation}\label{ScalProc}
\left(\overline{X}_N(t)\right)=\left(\overline{X}_{i}^N(t)\right)=\left(\frac{X_i^N(t)}{N}\right). 
\end{equation}

\subsection{The $\mathbf{M/M/\infty}$ queue}\label{MMISec}
We finish by recalling the definition of a process associated to a very simple, but important, CRN,
\[
\emptyset \mathrel{\mathop{\xrightleftharpoons[\mu]{\lambda}}} S_1.
\]
This is the ${M/M/\infty}$ queue with input parameter $\lambda{\ge}0$ and output parameter $\mu{>}0$. It is a  Markov process $(L(t))$  on $\N$ with transition rates
\[
x\longrightarrow
\begin{cases}
x{+}1&   \lambda \\
x{-}1&   \mu x.
\end{cases}
\]
The invariant distribution of $(L(t))$ is Poisson with parameter $\rho{=}\lambda/\mu$. For $t{\ge}0$, if $L(0){=}0$, then $L(t)$ has a Poisson distribution with parameter
\[
	\rho(1-e^{-\mu t}).
\]
If $T_N$ is the hitting time of $N$, $T_N{=}\inf\{t>0{:} L(t){\ge} N\}$, 
then the sequence
\[
\left(\frac{\rho^N}{(N{-}1)!}T_N\right)
\]
converges in distribution to an exponential random variable. As a consequence, for $p{\ge}1$,  the convergence in distribution
\[
\lim_{N\to+\infty} \left(\frac{L(N^pt)}{N}\right)=(0)
\]
holds. See Chapter~6 of~\citet{Robert}. 
This is an important process in the context of stochastic CRNs, see~\citet{LR23}.

\section{Scaling Properties}\label{GenSec}
In this section, we investigate some general properties of the asymptotic behavior of the sample paths of $(X(t))$ when the initial state is ``large''.

As it will be seen, this is a challenging problem in general. Sections~\ref{FNI} and~\ref{FNII} investigate a specific class of initial states of a network with {\em four} nodes for which a scaling description of the return path to $0$ involves, at the normal timescale, {\em jump processes} and not the nice solution of some set of ODES as it is usually the case. 

The scaling parameter $N$ can be thought, up to some fixed multiplicative constant, as the norm of the initial state. 
Our first result of this section, Proposition~\ref{InitH0}, shows that, for any $\eps{>}0$, the process goes ``quickly'' to a set of states for which the indices of the coordinates whose value is greater than $\eps N$ are at distance at least two, the other coordinates being $o(N)$.

The second result, Theorem~\ref{TheoImp}, considers the case when the number of nodes $m$ of the CRN is odd and the initial state is of the order of $(\alpha_1,0,\alpha_3,0,\ldots,0,\alpha_m)N$, with $\alpha_{k}{>}0$, for $k{\in}\{1,3,\ldots,m\}$. It is shown that, on the normal timescale, the scaled process converges in distribution to the solution of a system of ODEs. An averaging principle is proved to establish this convergence. Its proof  uses several results on the Markov process associated to a series of $M/M/\infty$ queues, Proposition~\ref{PropMMI} together with a coupling result, Proposition~\ref{PropCoup} for the proof of the tightness of occupation measures.  

The following proposition states essentially that, from this point of view, one can concentrate the study on asymptotic initial states such that the positive components have isolated indices. 
\begin{proposition}\label{InitH0}
 Under Condition~\eqref{IntCond1}, if $H_0{\steq{def}}\left\{x{\in}\R_+^m:x_ix_{i+1}=0, \forall 1{\le}i{<}m\right\}$, then  there exists $\beta{\in}H_0$ such that, for any $\eta{>}0$, there exists a stopping time $\tau_N$ satisfying the relation
  \[
\lim_{N\to+\infty}\P\left(\max_{1{\le}i{\le}m}\left(\rule{0mm}{4mm}\left|\frac{X_i^N(\tau_N)}{N}-\beta_i\right|\right)\ge \eta\right)=0,
\]
and the sequence $(\tau_N)$ is converging in distribution to $0$. 
\end{proposition}
In Section~\ref{TN}, there is an analogous result for the CRN with three nodes.  We show that there is a convenient stopping time, possibly depending on $N$, when the coordinates are arbitrarily close to the set of states with the ``correct'' orders of magnitude.
\begin{proof}
Define the scaled process
\[
\left(\widetilde{X}_N(t)\right)=\left(\frac{X_N(t/N)}{N}\right).
\]
Proposition~13 of~\citet{LR23} gives that the sequence of processes $(\widetilde{X}_N(t))$ is converging in distribution to $(x(t)){=}(x_i(t))$ the solution of the system of ODEs, for $2{<}i{<}m$, 
\begin{equation}\label{ODE1}
\begin{cases}
\dot{x_1}(t){=}{-}\kappa_2 x_{1}(t)x_{2}(t),  \; \dot{x_2}(t){=}{-}\kappa_3x_{2}(t)x_{3}(t), \\
\dot{x_i}(t){=}\kappa_{i-1}x_{i-2}(t)x_{i-1}(t){-}\kappa_{i+1}x_i(t)x_{i+1}(t),\\
\dot{x_m}(t){=}\kappa_{m-1}x_{m-2}(t)x_{m-1}(t).
\end{cases}
\end{equation}
We now show that $(x(t))$ converges to $\beta{\in}H_0$ as $t$ gets large. 

If at $t_0{>}0$, $\max_{i}|x_i(t_0){-}\beta_i|{<}\eta/2$, taking $\tau_N{=}t_0/N$, we have the desired result. 
For the convergence of $(x(t))$ to $\beta{\in}H_0$, we proceed by induction on the number of species. 
When $\alpha_1{=}0$, from Relations~\eqref{ODE1} we have $(x_1(t)){\equiv}(0)$.  Otherwise, if  $\alpha_1{>}0$ and $\alpha_2{=}0$, then, clearly, $(x_1(t)){\equiv}(\alpha_1)$  and $(x_2(t)){\equiv}(0)$. 

It is therefore enough to consider the case  $\alpha_1{>}0$ and $\alpha_2{>}0$. 
We have that $x_k(t){>}0$ for all $t{>}0$ and $k{\in}\{1, \ldots, m\}$. Indeed, this is true for $k{=}1$, $2$ and then,  by induction on $k{\in}\{3,\ldots,m\}$, for all the other indices.
Similarly, for all $2\leq k{<}m$, the functions
\[
\left(\sum_{i{\ge}0:2i+1\le k} x_{2i+1}(t)\right) \text{ and } \left(\sum_{i{\ge}1:2i\le k} x_{2i}(t)\right)
\]
are  (strictly) decreasing.  Again by induction on $1{\le}i{<}m$, the function $(x_i(t))$ has therefore a limit $\beta_i$ when $t$ goes to infinity and we also obtain the convergence of $(x_m(t))$ when $t$ goes to infinity, to a limit $\beta_m{>}0$. 
Relations~\eqref{ODE1} give the identity $\beta_i\beta_{i+1}{=}0$ for all $1{\leq}i{<}m$. The proposition is proved. 
\end{proof}
\begin{proposition}\label{PropMMI}
  For $p{\ge}1$,  let $(Y_i(t))$ be a Markov process on $\N^p$ with $Q$-matrix $R_{\lambda,\mu}{=}(r_{\lambda,\mu}(y,z))$ is such that, for $y{=}(y_i){\in}\N^p$ and $1{\le}i{\le}p$,
  \[
  \begin{cases}
    r_{\lambda,\mu}(y,y{+}e_1){=}\lambda, \quad r_{\lambda,\mu}(y,y{-}e_p){=}\mu_p y_p,\\
    r_{\lambda,\mu}(y,y{+}e_{i+1}{-}e_i){=}\mu_i y_i,
  \end{cases}
    \]
  where $e_i$ is the $i$th unit vector of $\N^p$ and $\mu{=}(\mu_i, i{=}1,\ldots,p)$ and $\lambda$ are positive constants. 
  \begin{enumerate}
  \item The invariant distribution of $(Y(t))$ is the product of  $p$ Poisson distributions with respective parameters $\lambda/\mu_i$;
  \item If, for $N{\ge}1$,  $(Y^N_i(t))$ is the Markov process with $Q$-matrix $R$ and initial state $(y^N_i)$ such that the sequence $(y^N_i/N)$ converges to $(0)$,
    \begin{itemize}
      \item For any $\eps{>}0$,
    \begin{equation}\label{CoupK}
 \sup_{N\ge 1}   \sup_{t{>}\eps}\E\left(\sum_{i=1}^pY^N_i(Nt)\right)<{+}\infty;
    \end{equation}
  \item The relation
    \[
    \lim_{N\to+\infty}\left(\frac{Y^N(Nt)}{N}\right)=(0)
    \]
    holds for the convergence in distribution. 
    \end{itemize}
  \end{enumerate}
\end{proposition}
\begin{proof}
The statement on the invariant distribution is straightforward to prove.

The Markov process can be described as a kind of Ehrenfest model with $p$ urns. The distribution of the sojourn time of a particle in the $i$th urn, $1{\le}i{\le}p$ is exponential with parameter $\mu_i$, after which it moves to urn $i{+}1$. The process $(Y_i^N(t))$ can be expressed as the sum of two independent processes, $(Y_{i,0}^N(t){+}Y_{i,1}^N(t))$, where $Y_{i,0}^N(t)$, resp. $Y_{i,1}^N(t)$ is the number of initial particles, resp. new particles (i.e. arrived after time $0$),  present in the $i$th urn at time $t$. The initial state of $(Y_{i,1}^N(t))$ is $(0)$ in particular. 

For $1{\le}i{\le}p$, we denote by $(E_{i,k}^0,k{\ge}1)$ and $(E_{i,k}^0,k{\ge}1)$  i.i.d. independent sequences of exponential random variables with parameter $\mu_i$. The process of the total  number of initial particles present in the system  has the same distribution as
\[
\left(\sum_{i=1}^p \sum_{k=1}^{y_i^N} \ind{E_{i,k}^0{+}E_{i+1,k}^0{+}\cdots{+}E_{p,k}^0\ge t}\right).
\]
The two assertions of 2) of the proposition for $(Y_{i,0}^N(t)$ are easily checked. 

If $(t_n)$ is a Poisson process on $\R_+$ with parameter $\lambda$, then the arrivals of particles at the $i$th urn has the same distribution as $(t_n{+}E_{1,n}{+}E_{2,n}{+}\cdots{+}E_{i-1,n})$, i.e. a Poisson process with parameter $\lambda$. Consequently,  the process $(Y_{i,1}^N(t))$ has the same distribution as the process of an $M/M/\infty$ queue of  Section~\ref{MMISec}, starting empty with arrival rate $\lambda$ and service rate $\mu_i$. 
The two assertions of 2) are a consequence of the properties of this model recalled in Section~\ref{MMISec}. 

\end{proof}

\begin{proposition}[Coupling]\label{PropCoup}
If $\kappa_0{=}0$ and   $m{=}2p{+}1$ and under Condition~\eqref{IntCond1}, if $\alpha{\in}H_0$ is such that $\alpha_{2i}{=}0$ for all $1{\le}i{\le}p$ and $\alpha_{2i{+}1}{>}0$, for all $0{\le}i{\le}p$, then there exist $\eta{>}0$ and  a coupling of  $(X_N(t)){=}(X_{k}^N(t),1{\le}k{\le}m)$ and $(Y_{i}(Nt),1{\le}i{\le}p)$, where $(Y(t))$ is a Markov process with $Q$-matrix $R_{\lambda,\mu}$ defined in Proposition~\ref{PropMMI}, for some $\lambda{>}0$, with $\mu{=}(\mu_i){=}\kappa_{2i+1}\eta$,  such that the relation
\begin{equation}\label{CoupIneq}
\sum_{k{=}1}^p X_{2k}^N(t)\le \sum_{k{=}1}^p Y_k^N(Nt)
\end{equation}
holds for all  $t{\le}T_N$, with
\begin{equation}\label{TN1}
T_N=\inf\left\{t{>}0: \exists i{\in}\{1,\ldots, p\}, X_{2i+1}^N(t){<}\eta N\right\}. 
\end{equation}
\end{proposition}
\begin{proof}
We fix  $\eta{=}\min\{\alpha_{2i+1}/2:1{\le}i{\le}p\}$ and $\lambda{>}0$ such that the relation
\[
X_{1}^N(0){+}X_{3}^N(0){+}\cdots{+}X_{2p+1}^N(0) \le \lambda N
\]
holds for all $N{\ge}1$. The condition $\kappa_0{=}0$ implies that $X_{1}^N(t){+}X_{3}^N(t){+}\cdots{+}X_{2p+1}^N(t)$ is also bounded by $\lambda N$ for all $t{>}0$.

We can represent the process $(Y_N(t)){=}(Y_i^N(t))$ as the solution of the following SDEs, for $2{\leq} i{\leq} p$: 
\[
	\begin{cases}
		\diff Y_1^N(t)&= {\cal P}_1 ((0, \kappa_1\lambda N), \diff t)- {\cal P}_3((0, \kappa_3\eta N Y_1^N(t{-})), \diff t),\\
		\diff Y_i^N(t)&= {\cal P}_{2i-1}((0, \kappa_{2i-1}\eta N Y_{i-1}^N(t{-})\eta N), \diff t)\\
		&\hspace{3cm} - {\cal P}_{2i+1}((0, \kappa_{2i+1}\eta N Y_i^N(t{-})), \diff t).
	\end{cases}
\]
with $(Y_N(0)){=}(X_{2i}^N(0)){=}(x_{2i}^N)$, where the $\cal{P}^i$ are introduced in the SDE~\eqref{SDEm}. 

For convenience, the two processes are described in terms of a queueing system, they are respectively referred to as the $X$-system for $(X_{2i}^N(t),1{\le}i{\le}2p{+}1)$ and the $Y$-system, for the process $(Y_i(t),1{\le}i{\le}p)$.  Initially there are $x_{2i}^N$ customers in the $i$th queue for both systems. 

External customers enter the system, at the rate $\kappa_1X_1^N(t)$ in the $X$-system, and $\kappa_1 \lambda N$ in the $Y$-system, therefore more customers enter the $Y$-system than the $X$-system. 
Furthermore, since we choose $\eta$ such that $X_{2i+1}(t){\ge} \eta N$ for all $1{\le}i{\le} m$, the service rate in every queue of the $Y$-system is smaller than the service rate in the $X$-system, meaning that a customer of the $Y$-system needs more time to run through the system than a customer of the $X$-system. This can be expressed with the following result: for $k\leq p$, for all $t\geq 0$, 
\begin{equation}\label{SCouple}
	S_{Y,k}^N(t)\steq{def}\sum_{i=1}^k Y_i^N(t) \geq \sum_{i=1}^k X_{2i}^N(t)\steq{def} S_{X,k}^N(t). 
\end{equation}

This can be proven by induction, using simple coupling arguments. The cases $k=1$ is straightforward. 
Assume that we have the result for $k-1$, the only Poisson processes that change $(S_k^N(t))$ are $\cal{P}_1$ and $\cal{P}_{2k+1}$.  Note that $\cal{P}_1$ is more likely for $(S_{Y,k}(t))$ than $(S_{X, k}(t))$, so that it preserves Relation~\eqref{SCouple}.

For $t{>}0$, if Relation~\eqref{SCouple} holds on $[0,t)$ and a jump of $\cal{P}_{2k+1}$ occurs at time $t$, there are  two possible situations:
\begin{itemize}
	\item If $S_{Y,k}(t{-})\geq S_{X,k}(t{-})+1$, the relation will still be valid after one jump. 
	\item If $S_{Y,k}(t{-})=S_{X,k}(t{-})$, since $S_{Y,k-1}(t{-}){\geq} S_{X,k-1}(t{-})$ by hypothesis of induction we have $Y_k(t{-}){\leq} X_{2k}(t{-})$, along with the fact that $X_{2k+1}(t{-}){\leq}\eta N$, we know that 
	\[
		\cal{P}_{2k+1}((0, \kappa_{2k+1} X_{2k}^N(t{-}) X_{2k+1}^N(t{-})), \diff t)\geq \cal{P}_{2k+1}((0, \kappa_{2k+1}\eta N Y_{k}^N(t{-})), \diff t), 
	\]
	which proves that Relation~\eqref{SCouple} is still true after the jump. 
\end{itemize} 
This concludes the proof of Proposition~\ref{PropCoup}. 
\end{proof}

\begin{theorem}\label{TheoImp}
If $m{=}2p{+}1$ and $\alpha{\in}H_0$ such that $\alpha_{2i}{=}0$ and $\alpha_{2i{-}1}{>}0$, for all $1{\le}i{\le}p{+}1$ then, under Condition~\eqref{IntCond1}, for the convergence in distribution, the relations
  \[
  \lim_{N\to+\infty} \left(\frac{X_{2i}^N(t)}{N},1{\le}i{\le}p\right){=}(0), 
    \lim_{N\to+\infty} \left(\frac{X_{2i+1}^N(t)}{N},0{\le}i{\le}p\right)=(\ell(t))=(\ell_{2i+1}(t))
  \]
  hold, where $(\ell(t))$ is the solution of the system of ODEs, for $1{\le}i{\le}p{-}1$,
\begin{equation}\label{ODEImp}
\begin{cases}
\dot{\ell}_{1}(t)&=  \displaystyle{-}\kappa_{1}\frac{\kappa_2}{\kappa_{3}}\frac{\ell_{1}(t)}{\ell_{3}(t)}\ell_{1}(t),\\
\dot{\ell}_{2i+1}(t)&=\displaystyle  \kappa_1\left(\frac{\kappa_{2i}}{\kappa_{2i+1}}\frac{\ell_{2i-1}(t)}{\ell_{2i+1}(t)}
{-}\frac{\kappa_{2(i+1)}}{\kappa_{2i+3}}\frac{\ell_{2i+1}(t)}{\ell_{2i+3}(t)}\right)\ell_{1}(t),\\
\dot{\ell}_{2p+1}(t)&=\displaystyle  \kappa_1\frac{\kappa_{2p}}{\kappa_{2p+1}}\frac{\ell_{2p-1}(t)}{\ell_{2p+1}(t)}\ell_{1}(t) {-}\kappa_{2p+2}\ell_{2p+1}(t),
\end{cases}
\end{equation}
with initial point $(\alpha_{2i+1})$. The function  $(\ell(t))$ converges to $(0)$ as $t$ goes to infinity. 
\end{theorem}
Hence from the large state $(\alpha_1,0,\alpha_3,0,\ldots,0,\alpha_m)N$, the process $(X^N(t))$ returns  on the normal timescale to a neighborhood of the origin along the curve $(\ell_1(t),0,\ell_3(t),0,\ldots,0,\ell_m(t))N$ of $\R_+^m$.
\begin{proof}
Since we are interested in the order of magnitude in $N$ of the components of the vector $(X_{2i+1}^N(t),0{\le}i{\le}p)$ on a finite time interval, the external arrivals with rate $\kappa_0$ do not play any role. Therefore, without a loss of generality, we can assume that $\kappa_0{=}0$, in particular $(X_{1}^N(t){+}X_{3}^N(t){+}\cdots{+}X_{2p+1}^N(t))$ is a non-increasing process by Relations~\eqref{SDEm}. The assumptions of Proposition~\ref{PropCoup} are therefore satisfied, the notations for $T_N$, $\eta$ and $\lambda$ in this proposition and its proof  are used. 

The proof is achieved in three steps. The main difficulty is controlling that the coordinates with an even index are $O(1)$ while the others are of the order of $N$.

\medskip
\noindent
{\sc Step~1. Tightness of the occupation measure}.\\
 Let  the {\em stopped occupation measure} $\mu_N$ on $\R_+{\times}\N^p$ is defined by
\[
\croc{\mu_N,f}=\int_{0}^{T_N} f(s,(X_{2i}^N(s),1{\le}i{\le}p))\diff s,
\]
for any $f{\in}{\cal C}_c(\R_+{\times}\N^p)$, where $T_N$ is defined by Relation~\eqref{TN1}. 
We now show that$(\mu_N)$ is tight for the convergence in distribution.

The tightness of  the sequence of random measures $(\mu_N)$  follows from Proposition~\ref{PropCoup}, Relation~\eqref{CoupK} of Proposition~\ref{PropMMI} and Lemma~1.3 of~\citet{Kurtz1992}. Lemma~1.4 of this reference gives that any limiting point of $\mu_\infty$ of a subsequence $(\mu_{N_k})$ can be expressed as
\begin{equation}\label{muI}
  \croc{\mu_\infty,f}=\int_{0}^{+\infty}\int_{\N^p}  f(s,x)\pi_s(\diff x) \diff s,
\end{equation}
for $f{\in}{\cal C}_c(\R_+{\times}\N^p)$,  where $(\pi_s)$ is an optional measure-valued process with values  in the subset of probability distributions on $\R_+{\times}\N^p$. 

If $f$ is a continuous function on $\R_+{\times}\N^p$ such that $|f(t,x)|{\le}C(t)\|x\|$ for a continuous function $(C(t))$, $x{=}(x_i){\in}\N^p$ and $\|x\|{=}x_1{+}\cdots{+}x_p$, then, for the convergence in distribution of processes, the relation
\begin{equation}\label{aux1}
\lim_{k\to+\infty}\left(\int_0^{T_{N_k}} f(s,(X_{2i}^{N_k}(s)))\diff s\right)=\left(\int_{0}^{+\infty}\int_{\N^p}  f(s,x)\pi_s(\diff x) \diff s\right).
\end{equation}
holds. The only (small) difficulty to show this result is the local behavior of $(X_N(t))$ at $t{=}0$. 
For $\eps{>}0$, we can show that 
\begin{equation}
\lim_{k\to+\infty}\left(\int_\eps^{T_{N_k}} f(s,(X_{2i}^{N_k}(s)))\diff s\right)=\left(\int_{\eps}^{+\infty}\int_{\N^p}  f(s,x)\pi_s(\diff x) \diff s\right).
\end{equation}
The proof is standard, by using the convergence of $(\mu_{N_k})$, the criterion of the modulus of continuity,  Theorem~7.3 of~\citet{Billingsley}, and an equivalent result to Relation~\eqref{CoupK}, which is easily proved: 
 \begin{equation}
	\sup_{N\ge 1}   \sup_{t{>}\eps}\E\left(\left(\sum_{i=1}^pY^N_i(Nt)\right)^2\right)<{+}\infty.
 \end{equation}
Now we have to show that for  $\eta$, $\eta'{>}0$, we can find $\eps{>}0$ and $k_0{\geq} 1$ such that for all $k{\geq}k_0$, 
\begin{equation}\label{eqClose0}
	\P\left(\left|\int_0^{\eps} f(s, (X_{2i}^{N_k}(s)))\diff s \right|>\eta \right)<\eta'. 
\end{equation}
We use the same notations as in the proof of Proposition~\ref{PropMMI}, and
\[
S_p^0\steq{def} E_{1,1}^0{+}\cdots{+}E_{p,1}^0.
\]
For $\eps{\in}(0,1)$ and  $\delta{>}0$, we have
\begin{multline*}
	\E\left(\left|\int_0^{\eps} f(s, (X_{2i}^{N_k}(s)))\diff s \right| \right)\leq\overline{C}\int_0^{\eps} \E(\|(X_{2i}(s))_i\|) \diff s\\
	\leq\overline{C}\lambda \E(S_p)\eps +\overline{C} \int_0^{\eps}\|(X_{2i}(0))_i\|\P\left( S_p^0\geq sN\right) \diff s\\
	\leq\overline{C}\lambda \E(S_p)\eps +\overline{C} \frac{\|(X_{2i}(0))_i\|}{\delta N}\E\left(e^{\delta S^0_p}\right)(1{-}e^{-\eps N\delta}),
\end{multline*}
with  $\overline{C}{=}\max\{C(t), t{\in} [0,1]\}$, which leads to Equation~\eqref{eqClose0}. By letting $\eps$ go to $0$, we obtain Relation~\eqref{aux1}.  

\medskip

\noindent
{\sc Step~2. A lower bound for $T_N$.}
    
For $1{\le}i{<}p$, let
\[
(D_{i}^N(t))\steq{def} \left(\int_{(0,t]}{\cal P}_{2i}\left(\left(0,\kappa_{2i}X_{2i-1}^N(s{-})X_{2i}^N(s{-})\right),\diff s\right)\right),
  \]
  then $D_i(t)$ is the number of molecules which have been transformed from $S_{2i-1}$ into $S_{2i+1}$ up to time $t$. When $t{<}T_N$, Proposition~\ref{PropCoup} and its proof give the relation
  \[
  D_i^N(t)\le   D_{Y,i}^N(t)\steq{def}\int_{(0,t]}{\cal P}_{2i}\left(\left(0,\kappa_{2i}\lambda N\|Y^N(Ns{-})\|\right),\diff s\right),
\]
Again standard arguments of stochastic calculus and the ergodic theorem for the Markov process $(Y(t))$  give the convergence in distribution
\[
\lim_{N\to+\infty} \left(\frac{D_{Y,i}^N(t)}{N}\right)=\left(\lambda^2\kappa_{2i}t\sum_{j=1}^p\frac{1}{\mu_j}\right).
\]
If $t_0$ is chosen so that for all $1{\le}i{<}p$, 
\[
	t_0{<} \frac{\alpha_{2i-1}}{2\lambda^2\kappa_{2i}\sum_{j=1}^p 1/\mu_j},
\]
and $t_0{<}\alpha_{2p+1}/ (2\lambda \kappa_{2p+2})$, then the sequence $(\P(T_N{\ge}t_0))$ converges to $1$. 

The SDEs for $(X_N(t))$ give the relation
\begin{multline}\label{aux2}
  \overline{X}_{2i+1}^N(t)=   \frac{x_{2i+1}^N}{N}+M_{i}^N(t)\\ {+}\kappa_{2i}\int_0^t  \overline{X}_{2i-1}^N(s)X_{2i}^N(s)\diff s
{-}\kappa_{2(i+1)}\int_0^t \overline{X}_{2i+1}^N(s)X_{2(i+1)}^N(s)\diff s,
\end{multline}
where $(M_{i}^N(t))$ is a martingale whose previsible increasing process is
\begin{multline*}
(\croc{M_{i}^N(t)})\\=\left(\frac{1}{N}\left(\kappa_{2i}\int_0^t  \overline{X}_{2i-1}^N(s)X_{2i}^N(s)\diff s{+}\kappa_{2(i+1)}\int_0^t \overline{X}_{2i+1}^N(s)X_{2(i+1)}^N(s)\diff s\right)\right).
\end{multline*}
On the time interval $[0,t_0]$, with Doob's Inequality,  since $\overline{X}_{2i\pm 1}^N(s){\le}\lambda$ for $s{\le}t_0$, the convergence in distribution~\eqref{aux1} gives that $(M_i^{N_k}(t))$ is converging in distribution to $(0)$. The criterion of the modulus of continuity,  Theorem~7.3 of~\citet{Billingsley},  with Relation~\eqref{aux2} shows that the sequence of processes $(\overline{X}^o_N(t)){\steq{def}}(\overline{X}_{2i+1}^N(t),0{\le}i{\le}p)$ is tight for the convergence in distribution on  $[0,t_0]$. Without loss of generality one can assume that the subsequence $(N_k)$ is such that  $(\mu_{N_k},(\overline{X}^o_{N_k}(t)))$ is converging to $(\mu_\infty,(\ell(t)))$, where $\mu_\infty$ has the representation~\eqref{muI}. 

\medskip

\noindent
    {\sc Step~3. Identification of $\mu_\infty$.}\\
In this part the convergence in distribution of processes refers implicitly to the time interval $[0,t_0]$. 
Let $g$  a function on $\N^p$ with finite support on $\N^p$. Relations~\eqref{SDEm} gives the relation, if $(X^e_N(t)){\steq{def}}(X_{2i}^N(t),1{\le}i{\le}p)$,
\begin{multline}\label{SDEg}
g\left(X^e_N(t)\right)=g\left(X^e_N(0)\right)+M_g^N(t)\\
{+}\kappa_{1}\int_0^tN\nabla_{e_{2}}(g)\left(X^e_N(s)\right)\overline{X}_{1}^N(s)\diff s\\
{+}\sum_{i=1}^{p}\kappa_{2i+1}\int_0^tN\nabla_{(e_{2(i+1)}-e_{2i})}(g)\left(X^e_N(s)\right)\overline{X}_{2i+1}^N(s)X_{2i}^N(s)\diff s,
\end{multline}
where, for $1{\le}k{\le}m{=}2p{+}1$, $e_k$ is the $k$th unit vector of $\R_+^m$, with the convention $e_{2(p+1)}{=}0$, and,  for $a{\in}\R_+^p$, 
\[
\nabla_a(g)(x)=g(x{+}a){-}g(x),\quad x{\in}\R_+^{p}. 
\]
With the same argument as for the martingale $(M_i^N(t))$ above, it is easily seen that $(M_g^{N_k}(t)/N)$ is converging in distribution to the null process. By dividing by $N$ Relation~\eqref{SDEg}, we obtain the relation
\begin{multline*}
\left(\int_0^t\int_{N^p}\nabla_{e_{2}}(f)\left(u\right)\kappa_1\ell_{1}^N(s)\pi_s(\diff u)\diff s\right.\\\left.
+\int_0^t\int_{N^p} \sum_{i=1}^{p}\nabla_{(e_{2(i+1)}-e_{2i})}(f)\left(u\right)\kappa_{2i+1}\ell_{2i+1}(s)u_i\pi_s(\diff u) \diff s\right)=(0),
\end{multline*}
which can be written as
\[
\left(\int_0^t\int_{N^p}R_{\lambda(s),\mu(s)}(f)(u)\pi_s(\diff u) \diff s\right)=(0),
\]
with, for $s{\ge}0$,  $\lambda(s){=}\kappa_1\ell_{1}^N(s)$ and $\mu(s){=}(\mu_i(s))$ and $\mu_i(s){=}\kappa_{2i+1}\ell_{2i+1}(s)$ for $1{\le}i{\le}p$. Step~2 gives that, almost surely for all  $s{\le}t_0$,  the relation $\ell_{2i{+}1}(s){\ge}\eta$ holds. From there, Proposition~\ref{PropMMI} and with standard arguments, as Lemma~1.5 of~\citet{Kurtz1992}, see also Section~B.2.2 of~\citet{LR23}, we obtain that, for $h{\in}{\cal C}_c([0,t_0]{\times}\N^p)$,
\[
\int_{\R_+{\times}\N^p} h(s,u)\mu_\infty(\diff s,\diff u)=\int_0^{+\infty} \int_{\N^p}h(s,u)\prod_{i=1}^p\Pois\left(\frac{\kappa_1\ell_{1}^N(s)}{\kappa_{2i+1}\ell_{2i+1}(s)}\right)(\diff u_i)\diff s,
\]
where $\Pois(\rho)$ is the Poisson distribution with parameter $\rho{>}0$. 

Relation~\eqref{aux2} on the subsequence $(N_k)$ gives, as $k$ go to infinity, the relation
\begin{align*}
\ell_{1}(t)&=  \alpha_{1} {-}\int_0^t  \kappa_{2}\ell_{1}(s)\frac{\kappa_1\ell_{1}(s)}{\kappa_{3}\ell_{3}(s)} \diff s,\\
\ell_{2i+1}(t)&=  \alpha_{2i+1} {+}\int_0^t  \kappa_{2i}\ell_{2i-1}(s)\frac{\kappa_1\ell_{1}(s)}{\kappa_{2i+1}\ell_{2i+1}(s)} \diff s\\
&\ \hspace{3cm}{-}\int_0^t \kappa_{2(i+1)}\ell_{2i+1}(s)\frac{\kappa_1\ell_{1}(s)}{\kappa_{2i+3}\ell_{2i+3}(s)}\diff s,\quad 1{\le}i{\le}p{-}1,\\
\ell_{2p+1}(t)&=  \alpha_{2p+1} {+}\int_0^t  \kappa_{2p}\ell_{2p-1}(s)\frac{\kappa_1\ell_{1}(s)}{\kappa_{2p+1}\ell_{2p+1}(s)} \diff s\\
    &\ \hspace{3cm} {-}\int_0^t \kappa_{2p+2}\ell_{2p+1}(s)\diff s,
\end{align*}
holds almost surely for all $t{\in}[0,t_0]$.

The convergence in distribution of $(\overline{X}^e_N(t))$ to $0$ on $[0,t_0]$ is a direct consequence of b) of Proposition~\ref{PropMMI}, Proposition~\ref{PropCoup} and Step~2.

Due to  local Lipschitz properties, the solution of the differential system~\eqref{ODEImp} can be defined on a maximal  time interval $[0,T_\infty)$, for some $T_\infty{\in}\R_+{\cup}\{{+}\infty\}$.  Assume that $T_\infty{<}{+}\infty$. For any $k{=}2i{+}1{\le}m$, $0{\le}i{\le}p$, by summing-up the $i{+}1$ first ODEs, we obtain that the function $(\ell_1(t){+}\ell_2(t){+}{\cdots}{+}\ell_{k}(t))$ is non-increasing on $[0,T_\infty)$, and, consequently, that $(\ell_k(t))$ has a limit, denoted as $\ell(T_\infty)$, when $t{\nearrow}T_\infty$.  The last ODE of Relations~\eqref{ODEImp}, we have $\dot{\ell}_m(t){\ge}{-}\kappa_{m+1}\ell_m(t)$, since $\alpha_{m}{=}\ell_{m}(0){>}0$, this implies that $\ell_m(T_\infty){>}0$. The same argument applied to the ODE for $(\ell_{2p-1}(t))$ gives that $\ell_{2p-1}(T_0){>}0$. By induction we obtain that the relation $\ell_{2i{+}1}(T_\infty){>}0$ holds for all $0{\le}i{\le}p$. This is a contradiction with that maximal property of $T_\infty$. We conclude that $T_\infty{=}{+}\infty$ and, therefore, that the convergence in distribution we have obtained holds in fact on $\R_+$. The fact that $\ell(\infty){=}(0)$ is proved with analogous arguments. 
The theorem is proved. 
\end{proof}

\section{The Three Species CRN}\label{TN}
In this section, we discuss in a simple setting the process associated to the CRN, starting from large initial states from a scaling perspective. As it will be seen, there are several scaling behaviors depending on the initial state. Proposition~\ref{30N0} shows that a first order limit associated to a class of initial states is a continuous but random function. This is due also to a boundary behavior of the law of mass action. 

\begin{figure}[ht]
  \centerline{
    \begin{tikzpicture}[->,node distance=1.7cm]
                        \node (O) [below] {$\emptyset$};
                        \node (A) [right of=O] {$S_1$};
                        \node (B) [right of=A] {$S_1{+}S_2$};
                        \node (C) [right of=B] {$S_2{+}S_3$};
                        \node (D) [right of=C] {$S_3$};
                        \node (E) [right of=D] {$\emptyset$};
                        \draw[-left to] (O)  -- node[above] {$\kappa_{0}$} (A);
                        \draw[-left to] (A)  -- node[above] {$\kappa_{1}$} (B);
                        \draw[-left to] (B)  -- node[above] {$\kappa_{2}$} (C);
                        \draw[-left to] (C)  -- node[above] {$\kappa_{3}$} (D);
                        \draw[-left to] (D)  -- node[above] {$\kappa_{4}$ } (E);
  \end{tikzpicture}}
\end{figure}

The set of SDEs for this network is
\begin{equation}\label{SDEm3}
\begin{cases}
\diff X_1^N(t)\displaystyle{=}{\cal P}_{0}\left(\left(0,\kappa_{0}\right),\diff t\right){-}{\cal P}_{2}\left(\left(0,\kappa_{2}X_{1}^N(t{-})X_{2}^N(t{-})\right),\diff t\right),\\
\diff X_2^N(t){=}{\cal P}_{1}\left(\left(0,\kappa_{1}X_{1}^N(t{-})\right),\diff t\right){-}{\cal P}_{3}\left(\left(0,\kappa_{3}X_{2}^N(t{-})X_{3}^N(t{-})\right),\diff t\right),\\
\diff X_3^N(t)={\cal P}_2((0,\kappa_2X_1^NX_2^N(t{-})),\diff t){-}{\cal P}_{4}((0,\kappa_{4}X_3^N(t{-})),\diff t),
\end{cases}
\end{equation}

With the convention that $N$, resp. $\emptyset$, stands for a quantity of the order of $N$, resp. $o(N)$, Proposition~\ref{InitH0} shows that the interesting initial states are ${\cal I}_1{=}(N,\emptyset,N)$, ${\cal I}_2{=}(\emptyset,\emptyset,N)$, ${\cal I}_3{=}(\emptyset,N,\emptyset)$, and ${\cal I}_4{=}(N,\emptyset,\emptyset)$. We will denote  ${\cal I}_0{=}(\emptyset,\emptyset,\emptyset)$. 

The subset ${\cal I}_1$ has been taken care of in Theorem~\ref{TheoImp}. For an initial state of  ${\cal I}_2$, with $X_3(0){\sim}x_3N$, on the time interval $[0,T]$,  the first two coordinates remain $o(N)$ and it is not difficult to see that the process $(X_3^N(t)/N)$ converges in distribution to $x_3\exp({-}\kappa_4 t)$. 

We now discuss the remaining cases, for simplicity we focus on two initial states $(0,N,0)$ in ${\cal I}_3$  and $(N,0,0)$ in  ${\cal I}_4$. 
\subsection{Initial State $\mathbf{(0,N,0)}$}\label{0N03}
The next result shows that the limit of $(X_2^N(t)/N)$ is a continuous function, but it is random, driven by an $M/M/\infty$ process. Note that since $(L(t))$ is converging in distribution to a Poisson random variable with parameter $\kappa_0/\kappa_3$, the ergodic theorem gives that $(\ell_2(t))$ is almost surely converging to $0$.
\begin{proposition}\label{30N0}
If $(X_N(t))$ is the solution of~\eqref{SDEm3} with initial condition  $(0, N,0)$, then the convergence in distribution
\[
\lim_{N\to+\infty}\left(\frac{X_1^N(t)}{N},\frac{X_2^N(t)}{N},\frac{X_3^N(t)}{N}\right)=(0,\ell_2(t)),0),
\]
holds, with
\[
\left(\ell_2(t)\right)\steq{def}\left(\exp\left({-}\kappa_3\int_0^t L(s)\diff s\right)\right),
\]
where $(L(t))$ has the same distribution as the  jump process associated to an $M/M/\infty$ queue with input rate $\kappa_0$ and service rate $\kappa_4$.
\end{proposition}
\begin{proof}
We give a sketch of the proof. Some of the arguments used are the same as in the proof of Theorem~\ref{TheoImp}. We fix some $T{>}0$ and $\delta{\in}(0,1)$, and set
  \[
  T_2^N\steq{def}\inf\left\{t{>}0: X_2^N(t)\le \delta N\right\},
  \]
Define $(Y(t)){=}(Y_1(t),Y_2(t))$ the Markov process on $\N^2$ of Proposition~\ref{PropMMI}, whose $Q$-matrix is $R_{\kappa_0,(\kappa_2,\kappa_4)}$. If $N_0$ is sufficiently large so that $\delta N_0{>}1$, then for $N{\ge}N_0$, in a similar way as in the proof of Proposition~\ref{PropCoup},  there exists a coupling of $(X_N(t))$ and $(Y(t))$ such that, for $t{<}T_2^N$, the relation
\[
X_1^N(t){+}X_3^N(t)\le Y_1(t){+}Y_2(t). 
\]
The SDEs~\eqref{SDEm3}  give that, for $t{\ge}0$, the relations
\begin{align}
X_1^N(t)&\leq {\cal P}_0((0,\kappa_0),(0,t]) \label{t1}\\
  \frac{X_2^N(t)}{N}&\le 1 {+}\frac{M_1^N(t)}{N}+\frac{\kappa_1}{N}\int_0^t {\cal P}_0((0,\kappa_0){\times}(0,s))\diff s,\label{t2}\\
  \frac{X_2^N(t)}{N}&\ge 1+\frac{M_2^N(t)}{N}-\kappa_3\int_0^t\frac{X_2^N(s)}{N}X_3^N(s)\diff s,\notag
\end{align}
hold, where $(M_i^N(t)/N)$, $i{=}1$, $2$, is a martingale which is converging in distribution to $0$.
For  any $\eps$ and $\eta{>}0$, the exists $N_0$ such that if $N{\ge}N_0$, then the relation
\begin{multline*}
\frac{X_2^N(t{\wedge}T_2^N)}{N}\ge 1{-}\eta-\kappa_3\int_0^{t{\wedge}T_2^N}\frac{X_2^N(s)}{N}X_3^N(s)\diff s
\\\ge 1{-}\eta-\kappa_3\int_0^{t{\wedge}T_2^N}(1{+}\eta)(Y_1(s){+}Y_2(s))\diff s
\ge 1{-}\eta-\kappa_3\int_0^{t}(1{+}\eta)(Y_1(s){+}Y_2(s))\diff s,
\end{multline*}
holds with probability at least $1{-}\eps$. Let
\[
\tau=\sup\left\{t{>}0:\int_0^{t}(Y_1(s){+}Y_2(s))\diff s\le \frac{1{-}\delta{-}\eta}{\kappa_3(1+\eta)}\right\}
\]
then
\[
\lim_{N\to+\infty} \P\left(\tau{\le}T_2^N\right)=1. 
\]
With the criterion of the modulus of continuity, it is easily seen that the sequence $(X_2^N(t{\wedge}\tau{\wedge}T)/N)$ is tight for the convergence in distribution. We can take  a subsequence associated to $(N_k)$ converging to some continuous process $(x_2(t))$.

With  the same notations as before,
\begin{equation}\label{t3}
\frac{X_2^N(t)}{N}=1+\frac{M_3^N(t)}{N}{+}\kappa_1\int_0^t\frac{X_1^N(s)}{N}\diff s-
\kappa_3\int_0^t \frac{X_2^N(s)}{N}X_3^N(s)\diff s,
\end{equation}
where $(M_3^N(t)/N)$ is a martingale which is converging in distribution to $0$. Relation~\eqref{t1} gives that, for the convergence in distribution, 
\[
\lim_{N\to+\infty} \left(\int_0^t\frac{X_1^N(s)}{N}\diff s\right)=(0). 
\]
We define $(L(t))$ as the solution of the SDE
\[
\diff L(t) = {\cal P}_{0}\left(\left(0,\kappa_{0}\right),\diff t\right){-}{\cal P}_{4}((0,\kappa_{4}L(t{-})),\diff t),
\]
Now using that, on the time  interval $[0,\tau{\wedge}T]$, the process $(X_2^N(t)/N)$ is greater than $\delta$, so that a species $S_1$ is transformed into a species $S_3$ after a duration of time stochastically bounded by an exponential distribution with parameter $\kappa_2\delta N$. Hence, the jumps ${+}1$ of the process $(X_3^N(t))$ are ``almost'' Poisson with rate $\kappa_0$, so that $(X_3^N(t))$ is ``close'' to $(L(t))$. This can be stated rigorously as follows. There are at most ${\cal P}_0((0,\kappa_0){\times}(0,T))$ jumps of size $1$ for  $(X_3^N(t))$, using Relation~\eqref{t2}, one gets that the variable
\[
\int_0^{T_2^N{\wedge}T} \frac{X_2^N(s)}{N}\left|X_3^N(s){-}L(s)\right|\diff s
\]
converges in distribution to $0$ since the duration between jumps of size ${+}1$ or ${-}1$ of both processes is converging to $0$. By taking the limit in Relation~\eqref{t3} along the sequence $(N_k)$ we get the identity
\[
(x_2(t),t\leq \tau{\wedge}T)= \left(\int_0^t x_2(s)L(s)\diff s,t\leq \tau{\wedge}T\right),
\]
which gives the desired convergence of $(X_2^N(t)/N)$ on the time  interval $[0,\tau{\wedge}T]$  and the representation of its limit. The procedure is repeated starting at time $\tau{\wedge}T$. The proposition is proved. 
\end{proof}

\subsection{Initial State $\mathbf{(N,0,0)}$}\label{N003}
When the initial state is $(N,0,0)$, guessing the time evolution is a little more tricky. One can see that $(X_2^N(t))$ grows quickly, at rate $\kappa_1N$ initially, and then $(X_3^N(t))$ grows at rate $\kappa_2 X_1^N(t)X_2^N(t)$. The problem is of understanding the correct orders of magnitude of $(X_2^N(t),X_3^N(t))$ since all these reactions interact. A (vague) intuition suggests that, quite quickly, $(X_3^N(t))$ is of the order of $N$ and that $(X_2^N(t))$ is $o(N)$, so that we are in the case $(N,\emptyset,N)$ already studied.

To show that this intuitive picture is ``correct'',  we can try to use a convenient scaling such as in the following proposition. 
\begin{proposition}
If the initial state is $(N, 0,0)$, then, for the convergence in distribution, 
	\[
		\lim_{N\rightarrow +\infty}\left(\frac{X_1^N(t/N^{2/3})}{N}, \frac{X_2^N(t/N^{2/3})}{N^{1/3}}, \frac{X_3^N(t/N^{2/3})}{N^{2/3}} , t\geq 0\right) =(1, x_2(t), x_3(t), t\geq 0)
	\]
	where $(x_2(t), x_3(t))$ is solution of the ODE
	\[
		\begin{cases}
			\dot{x}_2(t)= \kappa_1 -\kappa_3x_2(t)x_3(t)\\
			\dot{x}_3(t)= \kappa_2 x_2(t). 
		\end{cases}
	\]
\end{proposition}
We skip the proof of this result since the arguments are standard.  The solution of the above ODE is such that, as $t$ goes to infinity,  $x_2(t)$, resp.  $x_3(t)$ converges to $0$, resp. ${+}\infty$, which confirm our intuition. Note that the first coordinate of the process does not change at all. This is nevertheless not enough to obtain rigorously the correct orders of magnitude, for the $(X_3^N(t))$ in particular.   This shows that a scaling analysis, in the sense of deriving the convergence in distribution of scaled sample paths, is not always the best approach. 

Instead, with a combination of coupling arguments and convergence results, in a formulation similar as the one used in Proposition~\ref{InitH0}, we  can prove that at some instant, deterministic here, the coordinates of the process have the right order of magnitude. 

The following proposition shows that if the initial state is $(N,0,0)$,  then, at any time $t_0{>}0$, the state of the CRN satisfies the conditions of the initial state of Theorem~\ref{TheoImp}. It is in fact immediately in ${\cal I}_1$. 
\begin{proposition}
If $(X_N(t))$ is the solution of~\eqref{SDEm3} with initial condition $(N,0,0)$, then for any $\eps{>}0$ and $t_0{>}0$,  there exists $K{>}0$ and $\eta{>}0$ such that
  \[
		\liminf_{N\rightarrow +\infty}\P\left(\frac{X_1^N(t_0)}{N}, \frac{X_3^N(t_0)}{N}{\in}(\eta,1],X_2^N(t_0){\le}K \right)\ge 1{-}\eps.
  \]
\end{proposition}
\begin{proof}
  Since we are interested in the order of magnitude of $(X_1^N(t))$ in $N$ on finite time intervals, as before,  without loss of generality, we can assume that $\kappa_0{=}0$. In the rest of the proof, the first term of the right-hand side of the first relation of the SDEs~\eqref{SDEm3} is removed.  It is then easily seen that, in this case,  the process $(X_1^N(t){+}X_3^N(t))$ is non-increasing, in particular $X_3^N(t){\le}N$ for all $t{\ge}0$. 

For  $\delta{\in}(0,1)$, we  introduce the stochastic process $(Y_1^N(t),Y_2^N(t),Y_3^N(t))$, as the solution of the SDE
\begin{equation}
\begin{cases}
\diff Y_1^N(t)&= {-}{\cal P}_{2}((0,\kappa_{2}Y_1^NY_2^N(t{-})),\diff t)  \\
\diff Y_2^N(t)&={\cal P}_1((0,\kappa_1\delta N),\diff t) {-}{\cal P}_{3}((0,\kappa_{3}N Y_2^N(t{-})),\diff t)  \\
\diff Y_3^N(t)&={\cal P}_2((0,\kappa_2\delta NY_2^N(t{-})),\diff t){-}{\cal P}_{4}((0,\kappa_{4}Y_3^N(t{-})),\diff t),
\end{cases}
\end{equation}
with $Y_1^N(0){=}N$, $Y_2^N(0){=}Y_3^N(0){=}0$.

Define
\[
T_1^N=\inf\{t{\ge}0: X_1^N(t)\le \delta N\}.
\]
By induction on the jumps of the processes $(X_N(t),Y_N(t))$  on the time interval, it is not difficult to prove that, for any $t{<}T_1^N$, the relations $X_1^N(t){\ge}Y_1^N(t)$, $X_2^N(t){\ge}Y_2^N(t)$ and $X_3^N(t){\ge}Y_3^N(t)$ hold.

The process $(Y_2^N(t))$ has the same distribution as $(L(Nt))$ where $(L(t))$  is the process of an $M/M/\infty$ queue with arrival rate $\kappa_1\delta$ and service rate $\kappa_3$. Hence, with the ergodic theorem for positive recurrent Markov processes, we get that, for the convergence in distribution,
\[
\lim_{N\to+\infty}\left(\int_0^tf\left(Y_2^N(s)\right)\diff s\right)=\left(\croc{\Pois\left(\frac{\kappa_1\delta}{\kappa_3}\right),f}t\right),
\]
for any function $f$ on $\N$ with finite support. 

By using the classical approach to the proof of a stochastic averaging principle as presented in~\citet{Kurtz1992}, see also Section~B.2.2 of~\cite{LR23} for an example, we obtain the following convergence in distribution 
\[
\lim_{N\to+\infty} \left(\frac{Y_1^N(t)}{N},\frac{Y_3^N(t)}{N}\right)=\left(e^{-\kappa_1\kappa_2\delta t/\kappa_3},\frac{\kappa_1\kappa_2\delta^2}{\kappa_3\kappa_4}\left(1{-}e^{-\kappa_4 t}\right)\right). 
\]
For $t_0{>}0$, one can choose $\delta{\in}(0,1)$ sufficiently small so that $\delta{<}\exp({-}\kappa_1\kappa_2\delta t_0/\kappa_3)$, the above convergence shows that
\begin{equation}\label{qw1}
\lim_{N\to+\infty}\P\left(T_1^N{\ge}t_0\right)=1. 
\end{equation}
This concludes the proof of the lower bounds of $X_1^N(t_0)/N$ and $X_3^N(t_0)/N$. Remains to show the upper bound of $X_2^N(t_0)$. 

Define $\delta'{\in}(0,1)$ sufficiently small so that $\delta'{<} \kappa_1\kappa_2\delta^2\left(1{-}e^{-\kappa_4 t_0}\right)/(\kappa_3\kappa_4)$, and 
\[
T_3^N=\inf\{t{\ge}0: X_3^N(t)\le \delta' N\}.
\]
In a similarly way, as for Relation~\eqref{qw1}, we have 
\begin{equation}\label{qw2}
\lim_{N\to+\infty}\P\left(T_3^N{\ge}t_0\right)=1. 
\end{equation}
We introduce $(Z_2^N(t))$ the solution of the SDE,
\[
\diff Z_2^N(t)={\cal P}_1((0,\kappa_1 N),\diff t) {-}{\cal P}_{3}((0,\kappa_{3}\delta' N Z_2^N(t{-})),\diff t),
\]
with initial condition $Z_2^N(0){=}0$. As before, the process $(Z_2^N(t))$ has the same distribution as $(L^+(Nt))$ where $(L^+(t))$ is the process of an $M/M/\infty$ queue with arrival rate $\kappa_1$ and service rate $\kappa_3\delta'$.
For any $t_0{>}0$, $Z_2^N(t_0)$ converges in distribution to a Poisson distribution $\Pois(\rho)$ with parameter $\rho{\steq{def}}\kappa_1/(\delta'\kappa_3)$.

Since $X_1^N(t){\le}N$ for all $t{\ge}0$ and $X_3^N(t){\ge}\delta' N$ for $t{\le}T_3^N$, it is easily seen that $X_2^N(t){\le}Z_2^N(t)$ for $t{\le}T_3^N$. We conclude the proof of the proposition with Relation~\eqref{qw2}. 
\end{proof}

\section{AIMD processes: Invariant Distributions and a Limit Theorem}\label{AIMDSec}
We introduce two classes of AIMD stochastic processes (Additive Increase Multiplicative Decrease) in Sections~\ref{V1sec} and~\ref{V0sec} which play an important role  in the limit results of  Sections~\ref{FNI} and~\ref{FNII} for the CRN with four nodes. Section~\ref{PartSec} gives an averaging result where fast processes are AIMD processes, it will be used to establish the averaging principle of Section~\ref{FNII}. 

The first of these AIMD processes, $(R_1(t))$, describes the asymptotic behavior of the fourth coordinate $(X_4^N(t))$  on the timescale $(\sqrt{N}t)$ when  $(X_3^N(t))$ is $1$. The other one, $(R_0(t))$, is associated to the asymptotic behavior of $(X_4^N(t))$ on the timescale $(\sqrt{N}t)$.  The asymptotic time evolution  of $(X_2^N(\sqrt{N}t)/N,X_4^N(\sqrt{N}t)/\sqrt{N})$ can be expressed in terms of these AIMD processes.

In the averaging principle proved in Section~\ref{FNII}, the fast processes involved are AIMD processes. For this reason,  the asymptotic dynamic of the system is expressed in terms of  functionals of their invariant distributions. Sections~\ref{V1sec} and ~\ref{V1sec} give an explicit expression of the invariant distributions of $(R_1(t))$ and $(R_0(t))$.

Section~\ref{PartSec} establishes an asymptotic result for the time evolution of a particle system. This is a key ingredient in the proofs of limit theorems of Section~\ref{FNII}.

\begin{definition}\label{GamDef}
For $a$ and $b{>}0$,  $\Gamma_0(a,b)$ is the distribution on $\R_+$ with density
\[
\frac{b}{\Gamma(a)}(b x)^{a{-}1}e^{-b x}, \quad x{\ge}0.
\]

\end{definition}
The function $\Gamma$ is the classical Gamma function. See~\citet{Whittaker}.
The Laplace transform of  $\Gamma_0(a,b)$ at $\xi{\ge}0$ is  given  by
\[
\left(\frac{b}{b{+}\xi}\right)^a. 
\]
The  fractional moment of order $1/2$ of this distribution  is
\begin{equation}\label{MomGam}
\int_0^{+\infty}\sqrt{x}\,  \Gamma_0(a,b)(\diff x)=\frac{1}{\sqrt{b}\Gamma(a)}\int_0^{+\infty}x^{a{-}1/2}e^{-x}\diff x
=\frac{1}{\sqrt{b}}\frac{\Gamma(a{+}1/2)}{\Gamma(a)}.
\end{equation}

\subsection{The process $(R_1(t))$}\label{V1sec}
For   $\alpha$ and $\beta{>}0$, let ${\cal N}_1$ be a Poisson point process on $\R^2{\times}[0,1]$ with intensity measure $\alpha\diff s{\times}\diff t{\times} \beta a^{\beta-1}\diff a$. The point process  ${\cal N}_1$ can be represented as  ${\cal N}_1{=}(u_n,v_n,U_n^{1/\beta})$ where $(u_n,v_n)$ is a homogeneous Poisson point process with rate $\alpha$, independent of the i.i.d. sequence $(U_n)$ of uniformly distributed random variables on $[0,1]$. 

We now define $(R_1(t))$ as the solution of the SDE
\begin{equation}\label{AIMDeq1}
  \diff R_1(t) = \diff t +\int_{a{\in}[0, 1]} (a{-}1)R_1(t{-}){\cal N}_1((0,R_1(t{-})],\diff t, \diff a),
\end{equation}
with $R_1(0){=}v{\ge}0$.

The asymptotic behavior of $(R_1(t))$ is described in the following proposition. 
\begin{proposition}
  The process  $(R_1(t))$ converges in distribution to  the random variable $R_1(\infty)$. The distribution  of $R_1(\infty)^2$ is $\Gamma_0((\beta{+}1)/2,\alpha/2)$ of Definition~\ref{GamDef}.
\end{proposition}
The proof of the proposition uses  the embedded Markov chain of $(R_1(t))$.  
\begin{proof}
Let $(R_n)$ be the embedded Markov chain of $(R_1(t))$,  if $R_0{=}v{>}0$, then
\[
R_1= U^{1/\beta}\left(v{+} \tau_v\right),
\]
where $U$ and $\tau_v$ are independent random variables, $U$ with a uniform distribution on $[0,1]$ and $\tau_v$ is such that
\begin{equation}\label{tauv}
\int_0^{\tau_v} (v{+} s)\diff s =E_\alpha,
\end{equation}
where $E_\alpha$ is an exponentially distributed random variable with parameter $\alpha$, i.e. for $x{\ge}0$,
\[
\P(\tau_v\ge x)=\exp\left({-}\alpha\int_0^{x} (v{+} s)\diff s\right).
\]
We obtain that
\[
R_1^2= U^{2/\beta}\left(v^2{+}2E_\alpha\right).
\]
Now if $(U_i)$ and $(E_{\alpha,i})$ are independent i.i.d. sequences with the respective distribution of $U$ and $E_\alpha$, 
\begin{multline*}
W_1\steq{def} \sum_1^{+\infty} 2E_{\alpha,i}\prod_1^i U_k^{2/\beta}\\=
U_1^{2/\beta} \left( \sum_{i=2}^{+\infty} 2E_{\alpha,i}\prod_{k=2}^i U_k^{2/\beta}{+} 2E_{\alpha,1}\right)
=U_1^{2/\beta} \left(W_2{+} 2E_{\alpha,1}\right),
\end{multline*}
and $W_2{\steq{dist}}W_1$. The distribution of $\sqrt{W_1}$ is therefore invariant for the Markov chain $(R_n)$. The variable $W_1$ can be expressed as 
\[
W_1= \sum_1^{+\infty} 2E_{\alpha,i}\exp(-2 t_i/\beta),
\]
where, for $i{\ge}1$, $t_i{=}{-}\ln(U_1){-}\ln(U_2){-}{\cdots}{-}\ln(U_i)$.

From Proposition~1.11 of~\citet{Robert}, the marked point process ${\cal M}{=}(t_n,E_{\alpha,n})$ is Poisson with intensity $\diff u{\otimes}\alpha\exp(-\alpha  v)\diff v$ on $\R_+^2$ and, since 
\[
W_1= \int_{\R_+^2}2v\exp(-2 u/\beta){\cal M}(\diff u,\diff v),
\]
with Proposition~1.5 of~\cite{Robert} for the Laplace transform of ${\cal M}$, we get that, for $\xi{\ge}0$,
\begin{multline*}
\E\left(e^{-\xi W_1}\right)=
\exp\left({-}\int_{\R_+^2}\left(1{-}e^{-2\xi v\exp(-2 u/\beta)}\right)\alpha \exp(-\alpha  v)\diff v \diff u\right)\\
=\exp\left({-}\int_{0}^{+\infty}\frac{2\xi e^{-2 u/\beta}}{\alpha{+}2\xi e^{-2 u/\beta}}\diff u \right)
=\left(\frac{\alpha}{\alpha{+}2\xi}\right)^{\beta/2}.
\end{multline*}
The distribution of $W_1$ is $\Gamma_0(\beta/2,\alpha/2)$, hence with density
\[
\left(\frac{\alpha}{2}\right)^{\beta/2}\frac{x^{\beta/2{-}1}}{\Gamma(\beta/2)}e^{-\alpha x/2},\quad x{\ge}0.
\]
If $\nu_1$ denotes the invariant distribution of $(R_1(t))$, its representation in terms of the invariant distribution of the embedded Markov chain gives the relation
\begin{equation}\label{Invnu1}
\int_{\R_+}f(x)\nu_1(\diff x)=\frac{1}{\E(\tau_{\sqrt{W_1}})}\E\left(\int_0^{\tau_{\sqrt{W_1}}} f\left(\sqrt{W_1}{+} s\right)\diff s \right),
\end{equation}
for any non-negative Borelian function $f$ on $\R_+$

We take $f(x){=}\exp({-}\xi x)$ for some $\xi{\ge}0$,  Relation~\eqref{tauv} with $v{=}\sqrt{W_1}$ gives the relation
\begin{multline*}
  \E(\tau_{\sqrt{W_1}})\int_{\R_+}e^{-\xi x}\nu_1(\diff x)=\frac{1}{\xi }\left(\E\left(\left(e^{-\xi \sqrt{W_1}}\right)\right)-\E\left(e^{-\xi (\sqrt{W_1}{+}\tau_{\sqrt{W_1}})}\right)\right)\\
  =\frac{1}{\xi }\left(\E\left(e^{-\xi \sqrt{W_1}}\right)-\E\left(e^{-\xi (\sqrt{W_1{+}2E_\alpha})}\right)\right),
\end{multline*}
where $E_\alpha$ is an exponentially distributed random variable with parameter $\alpha$, independent of $W_1$.  
By using Fubini's formula, we have
\begin{multline}\label{Inv2}
\E\left(\tau_{\sqrt{W_1}}\right) \int_{\R_+}e^{-\xi x}\nu_1(\diff x)=\frac{1}{\xi }\left(\E\left(e^{-\xi \sqrt{W_1}}\right)-\E\left(e^{-\xi (\sqrt{W_1{+}2E_\alpha})}\right)\right)\\
=\int_0^{+\infty}e^{-\xi u}\P\left(\sqrt{W_1}{\le}u{\le}\sqrt{W_1{+}2E_\alpha}\right)\diff u,
\end{multline}
and, therefore, the density of $\nu_1$ at $u$ is proportional to 
\begin{multline*}
\P\left(\sqrt{W_1}{\le}u{\le}\sqrt{W_1{+}2E_\alpha}\right)=e^{-\alpha u^2/2}\E\left(e^{\alpha W_1/2}\ind{W_1\le u^2}\right)
\\=C_0e^{-\alpha u^2/2}\int_0^{u^2} e^{\alpha w/2} x^{\beta/2-1}e^{-\alpha w/2}\diff w
=C_1 e^{-\alpha u^2/2}u^{\beta},
\end{multline*}
where $C_0$ and $C_1$ are multiplicative constants. 
The proposition is proved. 
\end{proof}
\subsection{The process $(R_0(t))$}\label{V0sec}
For $\alpha{>}0$, let ${\cal N}_0$ be a Poisson point process on $\R{\times}\R_+$ with intensity measure $\alpha\diff t{\times}\exp(-a)\diff a$, it can be represented as the sequence of points $(u_n,E_n)$, where $(u_n)$ is a Poisson process on $\R_+$ with rate $\alpha$ and $(E_n)$ is an i.i.d. sequence of exponential random variables with parameter $1$, independent of $(u_n)$. 

For $\gamma$, $\beta{>}0$, we define $(R_0(t))$, the solution of the SDE
\begin{equation}\label{AIMDeq0}
  \diff R_0(t) = {-}\frac{1}{\gamma}R_0(t)\diff t +\int_{a{\in}\R_+} \left(\sqrt{R_0(t{-})^2{+}\frac{a}{\beta}}-R_0(t{-})\right){\cal N}_0(\diff t, \diff a),
\end{equation}
with $R_0(0){=}v{\ge}0$.

It should be noted that the integral expression in the above SDE is just a Dirac measure, since there is only one ``$a$'' when there is a jump in $t$. 
\begin{proposition}\label{VoInv}
  The process $(R_0(t))$ converges in distribution as $t$ goes to infinity to a random variable $R_0(\infty)$, such that  $R_0(\infty)^2$  has the distribution $\Gamma_0(\alpha \gamma/2,\beta)$ of Definition~\ref{GamDef}.
\end{proposition}
\begin{proof}
If $(W(t)){=}(R_0(t)^2)$, we  obtain easily  with a change of variable that the relation
\[
\diff W(t) = {-}\frac{2}{\gamma}W(t)\diff t +\frac{1}{\beta} \int_{a{\in}\R_+} a\, {\cal N}_0(\diff t, \diff a). 
\]
Its unique  solution is given by, for $t{\ge}0$,
\begin{multline*}
W(t)=W(0)e^{-2 t/\gamma}{+}\frac{1}{\beta}\int_0^t a e^{-2(t{-}s)/\gamma}{\cal N}_0(\diff s, \diff a)
\\\steq{dist} W(0)e^{-2 t/\gamma}{+}\frac{1}{\beta}\int_{0}^{t} a e^{-2s/\gamma}{\cal N}_0(\diff s, \diff a),
\end{multline*}
by reversibility and invariance of the Poisson process by translation. 
We obtain that $(W(t))$ converges in distribution to
\[
W_\infty\steq{def}\frac{1}{\beta}\int_{0}^{+\infty} a e^{-2 s/\gamma}{\cal N}_0(\diff s, \diff a).
\]
For $\xi{\ge}0$, by using the representation of the Laplace transform of  a Poisson process, see Proposition~1.5 of~\citet{Robert},  we get that
\begin{multline*}
\E\left(e^{-\xi W_\infty}\right)=
\exp\left(-\alpha \int_{\R_+^2}\left(1-\exp\left(-\frac{\xi}{\beta}a e^{-2 s/\gamma}\right)\right)e^{-a}\diff a\diff s\right)
\\=\exp\left(-\alpha \int_{0}^{+\infty}\frac{\xi  e^{-2s/\gamma}}{\beta{+}\xi  e^{-2 s/\gamma}}\diff s\right)
=\left(\frac{\beta}{\beta{+}\xi}\right)^{\alpha\gamma/2}.
\end{multline*}
We conclude that the distribution of $W_\infty$ is  $\Gamma_0(\alpha \gamma/2,\beta)$.
\end{proof}

\subsection{A Limiting Result on a Particle System}\label{PartSec}

We investigate the limiting behavior of the time evolution of a Markov process on $\N$ which is described in terms of a  particle system. A process of this type plays an important role in a time-changed version of the process $(U_4^N(t))$ in Section~\ref{FNII}. A limiting result is established, Theorem~\ref{LimitOccApp}, it plays a central role in the asymptotic analysis of the CRN with four nodes in Section~\ref{FNII}.

The kinetics are as follows: At time $t$, we are given $G$, an exponential random variable variable with parameter $\sigma$, and $(E_i)$ an  independent  i.i.d. sequence of exponential random variables  with parameter $\lambda$.
\begin{itemize}
\item Arrivals. A new particle arrives at rate $P_N(t{-})$, where $(P_N(t))$ is an adapted \cadlag process.
\item Departures. If there are $x{\in}\N$ particles, at rate $\delta x$, $\delta{>}0$, any particle $1{\le}i{\le}x$ such that $E_i{\le}G$ is removed.  
\end{itemize}
With the definition of Section~\ref{defnot}, ${\cal N}_\sigma$ is a Poisson marked point process  on the state space $\R_+^2{\times}\R_+{\times}\R_+^\N$ whose intensity measure is
\[
\diff s{\otimes}\diff t{\otimes}\sigma\exp({-}\sigma a)\diff a{\otimes}Q(\diff b),
\]
where $Q$ is the distribution on $\R_+^\N$ of an  i.i.d. sequence $(E_i)$ of exponential random variables with parameter $\lambda$.

The process of the number of particles, $(K_N(t))$,  is  defined as the solution of the SDE,
\begin{multline}\label{SDEBN}
\diff K_N(t)={\cal P}_0((0,P_N(t{-})),\diff t)\\{-}\int_{(a,b){\in}\R_+{\times}\R_+^\N} S(K_N(t-),a,b) {\cal N}\left(\left(0,\delta K_N(t{-})\right],\diff t,\diff a, \diff b\right),
\end{multline}
such that $K_N(0){=}w_N{\in}\N$. For $a{\ge}0$ and $b{=}(b_i){\in}\R_+^\N$, we denote
\begin{equation}\label{Sdef}
S(n,a,b)\steq{def}\sum_{i=1}^{n}\ind{b_i{\le} a}.
\end{equation}
It can be shown, with the criterion of the modulus of continuity, that the process  $(\overline{K}_N(t)){=}(K_N(t)/\sqrt{N})$ has convenient tightness properties on the timescale $(t/\sqrt{N})$. Since we are interested in the asymptotic behavior of this process on the normal timescale, we will investigate the asymptotic behavior of its occupation measure $\mu_N$ defined by
\begin{equation}\label{OccMA}
\croc{\mu_N,f}=\int_0^T F\left(s,\overline{K}_N(s)\right)\diff s=\int_0^T F\left(s,\frac{K_N(s)}{\sqrt{N}}\right)\diff s,
\end{equation}
for $F{\in}{\cal C}_c([0,T]{\times}\R_+)$.  
\begin{theorem}\label{LimitOccApp}
If  $(P_N(t))$ is a \cadlag adapted process on $\N$ such that $(P_N(t)/N)$ is a bounded process converging in distribution to a continuous positive process $(p(t))$ and  $(K_N(t))$ is the solution of  SDE~\eqref{SDEBN}, $K_N(0){=}w_N$ and the sequence $(w_N/N)$ is bounded, then the sequence of occupation measures $(\mu_N)$ defined by Relation~\eqref{OccMA}  converges in distribution to the measure $\mu_\infty$ defined by 
  \[
  \int_0^T \int_0^{+\infty} F\left(s,x\right)\mu_\infty(\diff s,  \diff x)
  =\int_0^T \int_0^{+\infty} F\left(s,\sqrt{x}\right)\Gamma_0\left(\frac{\sigma}{2\lambda}{+}\frac{1}{2}, \frac{\delta}{2p(s)}\right)(\diff x)\diff s,
\]
for $F{\in}{\cal C}_c([0,T]{\times}\R_+)$, where $\Gamma_0$ is the distribution of Definition~\ref{GamDef}.
\end{theorem}
\begin{proof}

The proof is done in several steps.

 \medskip
 \noindent
{\sc Step 1: Tightness properties of $(\mu_N)$.}\\

By integration of Relation~\eqref{SDEBN} and taking the expected value, we obtain that
\[
\delta\frac{\lambda}{\lambda{+}\sigma} \int_0^T \E\left(\overline{K}_N(s)^2\right)\diff s\le\frac{w_N}{N}{+}\int_0^T \E\left(\frac{P_N(s)}{N}\right)\diff s,
\]
hence
\begin{equation}\label{Fq2}
\sup_{N\ge 1}\int_0^T \E\left(\overline{K}_N(s)^2\right)\diff s <{+}\infty.
\end{equation}
This Relation will be used frequently in the proof of the different convergence theorems of Section \ref{FNII}.
Relation~\eqref{Fq2}, Lemmas~1.3 and~1.4 of~\citet{Kurtz1992} show that 1) the sequence $(\mu_N)$ is tight for the convergence in distribution, and 2)  if  $\mu_\infty$ is the limit of some converging subsequence $(\mu_{N_k})$, then it can be represented as
\[
\croc{\mu_{\infty},F}=\int_{[0,T]{\times}\R_+}F(s,x)\gamma_s(\diff x)\diff s,
\]
for $F{\in}{\cal C}_c([0,T]{\times}\R_+)$, where $(\gamma_s)$ is an optional process with values in probability distributions on $\R_+$.

If $f$ is a bounded continuous function on $\R_+$, then for $T{\ge}0$,
\[
\E\left(\int_0^T f\left(\overline{K}_N(s)\right)\overline{K}_N(s)\ind{\overline{K}_N(s){\ge}K}\diff s\right)\leq \frac{\|f\|_\infty}{K}\int_0^T\E\left(\overline{K}_N(s)^2\right)\diff s,
\]
hence, for the convergence in distribution
\begin{multline*}
\lim_{N\to+\infty} \int_0^T f\left(\overline{K}_N(s)\right)\overline{K}_N(s)\diff s\\=\int xf(x)\ind{s\le T} \mu_\infty(\diff s,\diff x)=\int_0^T \int_0^{+\infty} xf(x)\gamma_s(\diff x)\diff s.
\end{multline*}

The next steps are devoted to the identification of the limit. 

\medskip
\noindent
{\sc Step 2: Control of martingales.}\\
If $f$ is some continuous bounded function on $\R_+$, then,  for $t{\ge}0$, standard stochastic calculus gives the relation, for $t{\in}[0,T]$,
\begin{multline}\label{Fq1}
f(\overline{K}_N(t))=  f(\overline{K}_N(0))+ M_f^N(t)\\
+\int_0^t \left(f\left(\overline{K}_N(s){+}\frac{1}{\sqrt{N}}\right){-}f\left(\overline{K}_N(s)\right)\right) P_N(s)\diff s \\
+\int_0^t \left(\E\left(f\left.\left(\frac{1}{\sqrt{N}}\sum_{1}^{K_N(s)}\ind{E_i{>}G}\right)\right|{\cal F}_s\right){-}f\left(\overline{K}_N(s)\right)\right) \delta K_N(s)\diff s,
\end{multline}
where $(M_f^N(t))$ is a martingale. The expectation value of the last line is taken on the random variables $G$ and $(E_i)$.  If $f$ is ${\cal C}_1$ with compact support, its predictable increasing process $(\croc{M_f}(t))$ is such that
\begin{equation}\label{Fq3}
\croc{M_f}(t)\leq 4\|f'\|_\infty^2\int_0^t \frac{P_N(s)}{N}\diff s+4\delta\|f\|_\infty^2\int_0^t K_N(s)\diff s.
\end{equation}

With Relations~\eqref{Fq3} and~\eqref{Fq2}, we obtain that the sequence $(\E(\croc{M_f^N/\sqrt{N}}(T)))$ is converging to $0$ and, therefore, with Doob's Inequality, that the martingale $(M_f^N(t)/\sqrt{N})$ converges in distribution to $0$.

\medskip
\noindent
{\sc Step 3: A Technical estimate.}\\
This step is dedicated to the estimation of the height of the negative jumps of $(K_N(t))$. Let $f{\in}{\cal C}_c(\R_+)$,
and define $\Delta_N(f)$ as 
\[
\int_0^T \left|\E\left(f\left.\left(\overline{K}_N(s)\frac{1}{K_N(s)}\sum_{1}^{K_N(s)}\ind{E_i{>}G}\right){-} f\left(\overline{K}_N(s)e^{-\lambda G}\rule{0mm}{7mm}\right)\right|{\cal F}_s\right)\right|\overline{K}_N(s)\diff s.
\]
 The law of large numbers gives the relation
\[
\lim_{K_0\to+\infty} \sup_{x{\ge}0} \left|f\left(x\frac{1}{K_0}\sum_{1}^{K_0}\ind{E_i{>}G}\right){-} f\left(xe^{-\lambda G}\right)\right|=0
\]
which holds almost surely, and by Lebesgue's Theorem in $L_1$. The integrand of quantity $\Delta_N(f)$, with the factor $K_N(s)$ excluded,  is expressed as the sum of two terms. One with  the indicator function of the set $\{\overline{K}_N(s){\ge} \eta\}{=}\{K_N(s){\ge} \eta\sqrt{N}\}$, for some small $\eta{>}0$. The last estimate can then be used for this term. The other term with the set $\{\overline{K}_N(s){\le}\eta\}$ is negligible since the integrand is the difference of two expressions which are arbitrarily both close to $f(0)$. 

Relation~\eqref{Fq2} and Cauchy-Shwartz's Inequality give that, for $K_1{>}0$,
\[
\sup_{N}\P\left(\int_0^T\overline{K}_N(s)\diff s {\ge}K_1\right)\le \frac{T}{K_1}\sup_{N}\sqrt{\int_0^T \E\left(\overline{K}_N(s)^2\right)\diff s}.
\]
By combining these results, we obtain that the sequence $(\Delta_N(f))$ converges in distribution to $0$. 

\medskip
\noindent
{\sc Step 4: Conclusion}\\
If  $f{\in}{\cal C}_c^1$, by dividing Relation~\eqref{Fq1} by $\sqrt{N}$, and letting $N$ go to infinity, we obtain that, almost surely, the relation
\[
\int_0^t\int_{\R_+} p(s)f'(x)\gamma_s(\diff x) \diff s {+}\int_0^t \int_{\R_+^2}\left(f\left(xe^{-\lambda a}\right){-}f\left(x\right)\right)\delta x\sigma e^{-\sigma a}\diff a \gamma_s(\diff x)\diff s=0,
\]
holds for all $t{\ge}0$. Indeed it clearly holds almost surely for a fixed $t$ and,  therefore for all $t{\in}\Q$, consequently for all $t{\ge}0$ by continuity. We have established the relation, almost surely 
\[
\left(\int_0^t\croc{{\cal A}_{p(s)}(f),\gamma_s}\diff s\right)=(0),
\]
with, for $x{\ge}0$, $a{>}0$
\[
{\cal A}_a(f)(x)= af'(x) {+}\delta x\int_0^1 \left(f\left(x u^{\lambda/\sigma}\right){-}f\left(x\right)\right)\diff u.
\]
    Note that ${\cal A}_a$ is the infinitesimal generator of the Markov process $(R_1(a t))$, where $(R_1(t))$ is defined in Section~\ref{V1sec} with $\beta{=}\sigma/\lambda$ and $\alpha{=}\delta/a$. Hence, there exists a subset $S_f$ of $[0,T]$, negligible for the Lebesgue's measure, such that, almost surely, the relation $\croc{{\cal A}_{p(s)}(f),\gamma_s}{=}0$ holds, for all $s{\in}[0,T]{\setminus}S_f$. Since the set of  ${\cal C}_1$ functions with compact support on $\R_+$ is separable for the uniform norm, there exists $S_0$ of $[0,T]$, negligible for the Lebesgue's measure such that,  almost surely,  for any $f$ in  a dense subset of such functions the relation  $\croc{{\cal A}_{p(s)}(f),\gamma_s}{=}0$ for all $s{\in}[0,T]{\setminus}S_0$. Proposition~9.2 of~\citet{KurtzEthier} gives that for $s{\in}[0,T]{\setminus}S_0$, $\gamma_s$ is the invariant distribution of $(R_1(t))$. The proposition is proved. 
\end{proof}

\section{The Four Species CRN}\label{FNI}
We investigate the asymptotic behavior of the CRN with four nodes starting from an initial state of the form
$(0,N,0,0)$ for some large $N$. As explained in the introduction, we did not try a complete classification of initial states from the point of view of the asymptotic behavior of $(X_N(t))$ as we have done for $m{=}3$. We do believe however that this is {\em the} interesting  class of initial states, i.e. with a really unusual asymptotic behavior.

By using  Filonov's Theorem, see Theorem~6 of~\citet{LR23}, it can be proved, with some tedious but straightforward technicalities,  that the associated Markov process $(X(t))$  is positive recurrent. In particular, starting from $(0,N,0,0)$ the second coordinate $(X_2^N(t))$ will eventually decrease. The goal of this section and of Section~\ref{FNII} is of characterizing this decay. 

Up to now, we have seen that the ordinary timescale was enough to observe the decay of the norm of the process. See Theorem~\ref{TheoImp} and Sections~\ref{0N03} and~\ref{N003}. For this initial state, the situation is significantly different. It turns out that the convenient timescale is $(\sqrt{N}t)$ and that on this timescale the process lives in a set of states of the form $(a,y_N,b,v_N)$ with $a$, $b{\in}\N$ and $y_N$ and $v_N$ are respectively of the order of $N$ and $\sqrt{N}$.

We analyze the scaling properties of this CRN on the normal timescale in this section. It is shown that, with a scaling in space, the Markov process converges in distribution to a jump process under a  convenient topology on the space of \cadlag functions ${\cal D}([0,T],\R_+^4)$.  The limiting process is an AIMD process whose invariant distribution has been investigated in Section~\ref{V0sec}. The asymptotic behavior on the timescale  $(\sqrt{N}t)$ is analyzed in Section~\ref{FNII}. 

\begin{figure}[ht]
  \centerline{
    \begin{tikzpicture}[->,node distance=1.7cm]
                        \node (O) [below] {$\emptyset$};
                        \node (A) [right of=O] {$S_1$};
                        \node (B) [right of=A] {$S_1{+}S_2$};
                        \node (C) [right of=B] {$S_2{+}S_3$};
                        \node (D) [right of=C] {$S_3{+}S_4$};
                        \node (E) [right of=D] {$S_4$};
                        \node (F) [right of=E] {$\emptyset$};
                        \draw[-left to] (O)  -- node[above] {$\kappa_{0}$} (A);
                        \draw[-left to] (A)  -- node[above] {$\kappa_{1}$} (B);
                        \draw[-left to] (B)  -- node[above] {$\kappa_{2}$} (C);
                        \draw[-left to] (C)  -- node[above] {$\kappa_{3}$} (D);
                        \draw[-left to] (D)  -- node[above] {$\kappa_{4}$ } (E);
                        \draw[-left to] (E)  -- node[above] {$\kappa_{5}$ } (F);
  \end{tikzpicture}}
\end{figure}

The set of SDEs for the state $(X_N(t)){=}(X_i^N(t))$ network is 
\begin{equation}\label{SDEm4}
\begin{cases}
\diff X_1^N(t)&\displaystyle{=}{\cal P}_{0}\left(\left(0,\kappa_{0}\right),\diff t\right){-}{\cal P}_{2}\left(\left(0,\kappa_{2}X_{1}^NX_{2}^N(t{-})\right),\diff t\right),\\
\diff X_2^N(t)&{=}{\cal P}_{1}\left(\left(0,\kappa_{1}X_{1}^N(t{-})\right),\diff t\right){-}{\cal P}_{3}\left(\left(0,\kappa_{3}X_{2}^NX_{3}^N(t{-})\right),\diff t\right),\\
\diff X_3^N(t)&={\cal P}_2((0,\kappa_2X_1^NX_2^N(t{-})),\diff t){-}{\cal P}_{4}\left(\left(0,\kappa_{4}X_{3}^NX_{4}^N(t{-})\right),\diff t\right),\\
\diff X_4^N(t)&={\cal P}_3((0,\kappa_3X_2^NX_3^N(t{-})),\diff t){-}{\cal P}_{5}((0,\kappa_{5}X_4^N(t{-})),\diff t).
\end{cases}
\end{equation}

\subsection{Scaling Properties}\label{CycDef}
The rest of this section and Section~\ref{FNII} are devoted to the asymptotic behavior of $(X_N(t))$ when the initial state is $X_N(0){=}(0,y_N,0,v_N)$, with
\begin{equation}\label{InitIV}
\lim_{N\to+\infty} \left(\frac{y_N}{N},\frac{v_N}{\sqrt{N}}\right)=(y,v){\in}\R_+^2,\quad y{>}0. 
\end{equation}
The time evolution of the process $(X_N(t))$ is investigated, in a natural way, by a representation in terms of several steps of a cycle defined in terms of the points of ${\cal P}_0((0,\kappa_0],\diff t)$. We give a heuristic description of it for the moment. 
The cycle describes, with high probability,  the time evolution of the CRN in terms of the  values of $(X_1^N(t),X_3^N(t))$ with the successive states $(0,0)$, $(1,0)$, $(0,1)$ and $(0,0)$. This is not a formal definition but more an (hopefully) insightful picture of an important aspect of the kinetics of our CRN.
\subsection*{A Heuristic Description of a Cycle}
\begin{enumerate}
\item   If the initial state is $(0,y_N,0,v_N)$. Let $t_1$ be the first point of the  point process ${\cal P}_0((0,\kappa_0],\diff t)$. 
On the time interval $[0,t_1)$, only the last coordinate  $(X_4^N(t))$ changes, via the SDE
\begin{equation}\label{SDEA}
\diff A_N(t)={-}{\cal P}_{5}((0,\kappa_{5}A_N(t{-})),\diff t),
\end{equation}
with $A_N(0){=}v_N$. 

At time $t_1$ the state of the CRN is $(1,y_N,0,X_4^N(t_1))$.
\item\label{1y0}  If the initial state is $(1,y_N,0,v_N)$,  the variable $\tau_N^1$ is the time when  the $1$ at the first coordinate ``moves'' to the third coordinate. 
  At this instant,  on the event $\{t_1{>}\tau_N^1\}$,  the state becomes  $(0,y_N{+}Y_N,1,v_N{-}W_N)$, where
  \[
  Y_N={\cal P}_1((0,\kappa_1){\times}(0,\tau_N^1)) \text{ and }
  W_N=\int_0^{\tau_N^1}{\cal P}_5((0,\kappa_5X_4^N(s{-})),\diff s).
  \]
It is not difficult to see that the random variable $\tau_N^1$ is of the order of $1/N$ and, consequently, that  the sequence $(Y_N/N,W_N/\sqrt{N})$ is converging in distribution to $(0,0)$. At time $\tau_N^1$, the process starts at a state ``close'' to $(0,y_N,1,v_N)$.
\item   If the initial state is  $(0,y_N,1,v_N)$, the second coordinate may decrease until the time $\tau_N^2$ when the ``1'' of the third coordinate becomes $0$. The state is at this moment $(0,X_2^N(\tau_N^2),0,X_4^N(\tau_N^2))$ with high probability.
\end{enumerate}

In this approximate description, the possible values for the first and the third coordinates either $0$ or $1$. This turns out to be  essentially an accurate asymptotic  description of the CRN on the ``normal'' timescale~$(t)$ investigated in this section. A similar statement for the timescale~$(\sqrt{N}t)$ of Section~\ref{FNII} still holds but requires quite different arguments.

The distribution of the duration of the first step is exponential with parameter $\kappa_0$. The decay of $(X_4^N(t))$ occurs essentially during this step. Step~(2) is a (short) transition, mentioned only to have a straight formulation of the limit results. 

The third step is when $(X_2^N(t))$ decreases and $(X_4^N(t))$ builds up. As it will be seen its duration is $O(1/\sqrt{N})$ and the number of jumps of the process during this step is large, of the order of $\sqrt{N}$. This feature has a significant impact on the statements of the scaling results for $(X_2^N(t)/N,X_4^N(t)/\sqrt{N})$  via the topologies used on the space of \cadlag functions.

With a time change, the third step can be ``removed'' and  a convergence result holds for the usual Skorohod topology, the $J_1$-topology. See Proposition~\ref{PropJ1}.  Otherwise, on the full timescale, the $M_1$-Skorohod topology or the $S$-topology has to be used, see Proposition~\ref{PropM1}. See~\citet{Whitt} and~\citet{Jakubowski} for general presentations of these topologies.

In any of these cases, the limiting process is an AIMD process analyzed in Section~\ref{V0sec}. See Figure~\ref{fig4-1}. 
We start with the main limiting result for the third step. 
\begin{proposition}\label{Hconv}
If  $X_N(0){=}(0,y_N,1,v_N)$ satisfy Relation~\eqref{InitIV} and 
  \[
  \tau_N^2\steq{def}\inf\left\{t{>}0: X_3^N(t)=0\right\},
  \]
  then,  under Condition~\eqref{InitIV}, the relation
  \begin{equation}\label{CVj}
  \lim_{N\to+\infty}\left(\sqrt{N}\tau_N^2,\frac{X_2^N(\tau_N^2)}{N},\frac{X_4^N(\tau_N^2)}{\sqrt{N}}\right)=
  (H_{y,v},y,v{+}\kappa_3y H_{y,v}),
  \end{equation}
holds for the convergence in distribution,   where $H_{y,v}$ is a non-negative random variable whose distribution is given by
\begin{equation}\label{Hy}
\E_y(f(H_{y,v}))=\int_0^{+\infty} f\left(\frac{1}{\kappa_3y}\left(\sqrt{v^2{+}2\frac{\kappa_3}{\kappa_4}ys}- v\right)\right)e^{-s}\diff s,
\end{equation}
  for any function $f{\in}{\cal C}_c(\R_+)$. 
\end{proposition}
\begin{proof}
  Let $(Y_2^N(t),Y_4^N(t))$ be the solution of the SDEs
  \[
\begin{cases}
\diff Y_2^N(t)&{=}{-}{\cal P}_{3}\left(\left(0,\kappa_{3}Y_{2}^N(t{-})\right),\diff t\right),\\
\diff Y_4^N(t)&={\cal P}_3((0,\kappa_3Y_2^N(t{-})),\diff t){-}{\cal P}_{5}((0,\kappa_{5}Y_4^N(t{-})),\diff t),
\end{cases}
\]
with the initial condition $(Y_2^N(0),Y_4^N(0)){=}(N,v_N)$. Standard stochastic calculus as in Section~\ref{TN} gives  the convergence in distribution as processes for the uniform topology
\begin{equation}\label{eqa2}
  \lim_{N\to+\infty}\left(\frac{Y_2^N(t/\sqrt{N})}{N},\frac{Y_4^N(t/\sqrt{N})}{\sqrt{N}}\right)
  =(y,v{+}\kappa_3yt).
\end{equation}
We define
\[
  \tau^Y_N=\inf\left\{t{>}0: \int_0^t{\cal P}_{4}\left(\left(0,\kappa_{4}Y_{4}^N(s{-})\right),\diff s\right){\ne}0\right\}.
\]
Let $E_1$ be an  exponential random variable with parameter $1$, independent of ${\cal P}_3$ and ${\cal P}_5$. It is easily seen that if $H_N$ is the solution of the relation
\begin{equation}\label{eqa1}
  \kappa_4 \int_0^{H_N} Y_4^N(s)\diff s=E_1
\end{equation}
  then the relation
  \[
  (\tau^Y_N,Y_2^N(\tau^Y_N),Y_4^N(\tau^Y_N))\steq{dist}  (H_N,Y_2^N(H_N),Y_4^N(H_N))
  \]
  holds. A change of variable in Relation~\eqref{eqa1} gives the identity
  \[
  \kappa_4 \int_0^{\sqrt{N}H_N} \frac{Y_4^N(s/\sqrt{N})}{\sqrt{N}}\diff s=E_1.
  \]
  Relation~\eqref{eqa2} shows that the sequence $(\sqrt{N}H_N)$ is tight and also that any of its limiting points $H$ satisfies the relation
  \[
 \kappa_4\left(vH+\frac{\kappa_3}{2}yH^2\right)\steq{dist} E_1,
  \]
  and therefore the convergence in distribution of this sequence to $H_{y,v}$.  The relation
\begin{multline*}
  \left(\sqrt{N}H_N,\frac{Y_2^N(H_N)}{N},\frac{Y_4^N(H_N)}{\sqrt{N}}\right)\\
  =  \left(\sqrt{N}H_N,\frac{Y_2^N({\sqrt{N}H_N}/{\sqrt{N}})}{N},\frac{Y_4^N({\sqrt{N}H_N}/{\sqrt{N}})}{\sqrt{N}}\right)
\end{multline*}
and  the convergence~\eqref{eqa2} show that, when $N$ goes to infinity these random variables converge in distribution to the right-hand side of Relation~\eqref{CVj}. 

Until time $t_1{\wedge}\tau_N^2$, it is easy to see that $(X_2^N(t), X_4^N(t)){\steq{dist}}(Y_2^N(t), Y_4^N(t))$. The proof is concluded by noting that 
\[
	\lim_{N\rightarrow +\infty} \P\left( t_1< \tau_N^Y\right) = 0,
\] 
since $t_1$ is a exponential random variable with parameter $\kappa_0$, independent of $\cal{P}_3$ and $\cal{P}_5$. 
\end{proof}
We complement the last proposition with a technical corollary which describes the time evolution on the  time interval $[0,\tau_2^N)$ of the process $(X_2^N(t),X_4^N(t))$. It is used in the proof of the convergence for the $M_1$-topology of Section~\ref{NTS}.
\begin{corollary}\label{Corol1}
With the notations and Assumptions of Proposition~\ref{Hconv}, if $T{>}0$, for the convergence in distribution
  \[
  \lim_{N\to+\infty}\left(\frac{X_2^N(t/\sqrt{N})}{N},\frac{X_4^N(t/\sqrt{N})}{\sqrt{N}}, t{\le}\sqrt{N}\tau_2^N\right)=
  ((y,v{+}\kappa_3 y t),  t{\le}H_{y,v}).
  \]
\end{corollary}
\begin{proof}
  This is a consequence of a) the coupling  of the proof of   Proposition~\ref{Hconv}, the relation $(X_2^N(t),X_4^N(t)){=}(Y_2^N(t),Y_4^N(t))$ holds for $t{<}\tau_2^N\wedge t_1$,  and b) that Relation~\eqref{eqa2} can be strengthened  as
  \[
    \lim_{N\to+\infty}\left(\left(\frac{Y_2^N(t/\sqrt{N})}{N},\frac{Y_4^N(t/\sqrt{N})}{\sqrt{N}}\right), \sqrt{N}\tau_2^N\right)
  =((y,v{+}\kappa_3yt), H_{y,v}).
  \]
\end{proof}
We now return to the investigation of the asymptotic behavior of
\[
 \left(\frac{X_2^N(t)}{N},\frac{X_4^N(t)}{\sqrt{N}}\right),
\]
when the initial state is such that $(X_2^N(0),X_4^N(0)){=}(y_N,v_N)$ and Relation~\ref{InitIV} holds, and $X_1^N(0)$, $X_3^N(0){\in}\{0,1\}$. We first show that, up to a time change, there is indeed a convergence in distribution for the $J_1$-topology. See Proposition~\ref{PropJ1}. Without a time change, there is a convergence in distribution but for weaker topologies, the $M_1$-topology and the $S$-topology. See the discussion in Section~\ref{NTS}. 

\subsection{Convergence with a Random Time Change}\label{TS}
The time change considered in this section consists in removing the instants $t$ of step~(3) of Definition~\ref{CycDef}, i.e. when $X_3(t){\ne}0$. We introduce, for $t{\ge}0$, 
\[
L_0^N(t)=\int_0^t \ind{X_3^N(s)=0}\diff s \quad\text{and}\quad  \ell_0^N(t)=\inf\{s{\ge}0: L_0^N(s) > t\}. 
\]
\begin{figure}[ht]
\scalebox{0.7}{\includegraphics{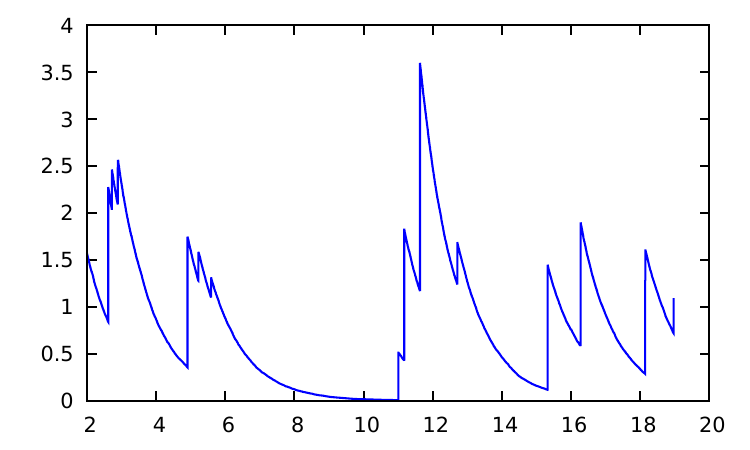}}
\put(-260,120){$\frac{X_4^N}{\sqrt{N}}$}
\put(-5,10){$t$}
\caption{CRN with Four Nodes.\\ Simulation: A snapshot of $(X_4^N(t)/\sqrt{N})$.\\Initial state $(0,N,0,0)$, $\kappa_i{=}1$, $i{=}0,$\ldots,$5$ and $N{=}9{\cdot}10^6$.}\label{fig4-1}
\end{figure}
\begin{proposition}\label{PropJ1}
  If  $X_N(0){=}(0,y_N,1,v_N)$ satisfy Relation~\eqref{InitIV}, then for the convergence in distribution for the $J_1$-Skorohod topology,
  \[
  \lim_{N\to+\infty} \left(\frac{X_2^N(\ell_0^N(t))}{N},\frac{X_4^N(\ell_0^N(t))}{\sqrt{N}}\right)=(y,V_y(t)),
  \]
  where $(V_y(t))$ is the Markov process on $\R_+$ whose infinitesimal generator ${\cal A}_y$ is given by, for $f{\in}{\cal C}_c^1(\R_+)$ and $x{\in}\R_+$, 
  \[
  {\cal A}_y(f)(x) = -\kappa_5 x f'(x)+\kappa_0\int_0^{+\infty}\left(f\left(\sqrt{x^2{+}2\frac{\kappa_3}{\kappa_4}ys}\right){-}f(x)\right)e^{-s}\diff s
\]
\end{proposition}
The process $(V_y(t))$ is in  the class of AIMD processes introduced in Section~\ref{V0sec}. 
See Figure~\ref{fig4-1} for an illustration of a sample path of $(V_y(t))$. 
\begin{proof}
 The result itself is quite intuitive in view of Proposition~\ref{Hconv}. Some care is nevertheless necessary in order to deal with  formal aspects of the $J_1$-topology. The proof is direct. One of its simple ingredients is that the convergence 
\begin{equation}\label{J1simp}
  \lim_{N\to+\infty} \left(x_N\ind{a_N\le t}{+}y_N\ind{b_N\le t}\right)=\left(x\ind{a\le t}{+}y\ind{b\le t}\right)
\end{equation}
holds for the $J_1$-topology if the sequences $(x_N)$, $(y_N)$,$(a_N)$,$(b_N)$ converge respectively to $x$, $y$, $a$, $b$ with $a{\ne}b$. Extensions with more terms also hold. See Section~VI.1 of~\citet{Jacod} for example. The proof also use Skorohod's representation theorem repeatedly. See Theorem~1.8 of Ethier and Kurtz~\cite{KurtzEthier}.

We denote by $(t_n)$ the non-decreasing sequence of points of ${\cal P}_0((0,\kappa_0],\diff t)$, with the convention that $t_0{=}0$. An important point is that this sequence is not depending on the scaling parameter $N$. This property will simplify the proofs of convergence in distribution of this section. We fix $T{>}0$, as usual ${\cal D}([0,T])$ denotes the space of \cadlag functions on $[0,T]$, we will have to consider the events,  $T{\in}[t_n,t_{n+1})$, $n{\ge}1$.  

Let $(A_0^N(t))$ be the solution of SDE~\eqref{SDEA} with $A_0^N(0){=}v_N$. A  standard argument gives the convergence in distribution
\begin{equation}\label{eqZ1}
\lim_{N\to+\infty} \left(\frac{A_0^N(t)}{\sqrt{N}}, 0{\le}t{<}T\right)=\left(ve^{-\kappa_5 t}, 0{\le}t{<}T\right),
\end{equation}
for the uniform topology on ${\cal D}([0,T])$. By using Skorohod's representation theorem, see Theorem~1.8 of Ethier and Kurtz~\cite{KurtzEthier}, one can assume that there exists a probability space on which the sequence of processes $(A_0^N(t))$ are defined and the convergence~\eqref{eqZ1} holds almost surely for the uniform norm.

We use similar notations as in Definition~\ref{CycDef} 
\[
  \tau_N^1\steq{def}\inf\left\{t{>}0: X_3^N(t_1{+}t)=1\right\},
  \tau_N^2\steq{def}\inf\left\{t{>}0: X_3^N(t_1{+}\tau_N^1{+}t)=0\right\}.
  \]

On the event $\{t_1{\le} T\}$, the processes $((y_N,A_0^N(t)),t{<}t_1)$ has the same distribution as $((X_2^N(t),X_4^N(t)),t{<}t_1)$ and, almost surely
\[
\lim_{N\to+\infty}\frac{A_0^N(t_1)}{\sqrt{N}}= \overline{v} \steq{def} ve^{-\kappa_5 t_1}.
\]
then, by Proposition~\ref{Hconv}, the convergence in distribution 
\begin{multline}\label{eqH1}
  \lim_{N\to+\infty}\left(\sqrt{N}\tau_N^1,\sqrt{N}\tau_N^2,\frac{X_2^N(t_1{+}\tau_N^1{+}\tau_N^2)}{N},\frac{X_4^N(t_1{+}\tau_N^1{+}\tau_N^2)}{\sqrt{N}}\right)\\
  =  (0,H_{y,\overline{v}},y,\overline{v}{+}\kappa_3y H_{y,\overline{v}}),
\end{multline}
holds, where
\begin{equation}\label{eqH2}
V_1 \steq{def}\overline{v}{+}\kappa_3y H_{y,\overline{v}}\steq{dist} \sqrt{\overline{v}^2{+}2\frac{\kappa_3}{\kappa_4}yE_1}, 
\end{equation}
where $E_1$ is an exponentially distributed random variable with parameter $1$ by Relation~\eqref{Hy}. Using again    Skorohod's representation theorem, it can be assumed that, with a convenient probability space, the convergence~\eqref{eqH1} holds almost surely.

Let $v_1^N{=}X_4^N(t_1{+}\tau_N^1{+}\tau_N^2)$ and  $(A_1^N(t))$ the solution of~\eqref{SDEA} associated to an independent Poisson process ${\cal P}_5$ and with initial point $v_1^N$. The convergence
\[
\lim_{N\to+\infty} \left(\frac{A_1^N(t)}{\sqrt{N}}, 0{\le}t{<}T\right)=\left(V_1e^{-\kappa_5 t}, 0{\le}t{<}T\right)
\]
holds and, up to a change of probability space, it holds almost surely for the uniform topology.

Our time change $(\ell_0^N(t))$ ``removes'' the time interval $[t_1{+}\tau_N^1, t_1{+}\tau_N^1{+}\tau_N^2)$.
We define
\begin{align*}
&\left(\widetilde{X}_2^N(t),\widetilde{X}_4^N(t)\right)=
\ind{t{<}t_1{+}\tau_N^1}\left(y_N,A_0^N(t)\right)\\
    &\hspace{4cm}{+}\ind{t_1{+}\tau_N^1{\le}t}\left(X_2^N(t_1{+}\tau_N^1{+}\tau_N^2),A_1^N(t{-}t_1{-}\tau_N^1)\right)\\
&=\left(y_N,A_0^N(t)\right){+}\ind{t_1{+}\tau_N^1{\le}t}\left(X_2^N(t_1{+}\tau_N^1{+}\tau_N^2){-}y_N,A_1^N(t{-}t_1{-}\tau_N^1){-}A_0^N(t)\right).
\end{align*}
Relation~\eqref{J1simp} gives that, for the $J_1$-topology,
\[
\lim_{N\to+\infty} \left(\frac{\widetilde{X}_2^N(t)}{N},\frac{\widetilde{X}_4^N(t)}{\sqrt{N}}\right)=
(y,V_y(t))\steq{def}\left(y,\ind{t{<}t_1}ve^{-\kappa_5 t}{+}\ind{t_1{\le}t}V_1e^{-\kappa_5 (t-t_1)}\right).
\]
Note that on the event $t_1{\le}T{<}t_2$, we have
\[
\left(\frac{\widetilde{X}_2^N(t)}{N},\frac{\widetilde{X}_4^N(t)}{\sqrt{N}}\right)\steq{dist}
\left(\frac{X_2^N(\ell_0^N(t))}{N},\frac{X_4^N(\ell_0^N(t))}{\sqrt{N}}\right)
\]
if the second component of $(X_N(t))$ does not increase on $[t_1,t_1{+}\tau_N^1]$, i.e. on the event
${\cal E}_N=\left\{{\cal P}_1((0,\kappa_1){\times}[t_1,t_1{+}\tau_N^1])=0\right\}$. The probability of ${\cal E}_N$  is arbitrarily close to $1$ as $N$ gets large since $(\sqrt{N}\tau_N^1)$ converges to $0$. Hence if $\Phi$ is a bounded continuous functional on ${\cal D}([0,T])$ endowed with the $J_1$-topology, we obtain the relation
\begin{multline*}
  \lim_{N\to+\infty} \E\left(\Phi\left(\left(\frac{X_2^N(\ell_0^N(t))}{N},\frac{X_4^N(\ell_0^N(t))}{\sqrt{N}}\right)\right)\ind{t_1{\le}T{<}t_2}\right)
\\=  \E\left(\Phi\left(\left(y,V_y(t)\right)\right)\ind{t_1{\le}T{<}t_2}\right).
\end{multline*}
note that $V_y(t_1{-}){=}\overline{v}$ and
\[
V_y(t_1){-}V_y(t_1{-}){=}\sqrt{V_y(t_1{-})^2{+}2\frac{\kappa_3}{\kappa_4}yE_1}{-}V_y(t_1{-}),
\]
by Relation~\eqref{eqH2}. We conclude that on the event $\{t_1{\le}T{<}t_2\}$, the processes
\[
\left(\frac{X_2^N(\ell_0^N(t))}{N},\frac{X_4^N(\ell_0^N(t))}{\sqrt{N}}, t\in [0, T]\right)
\]
converge in distribution for the $J_1$-topology to $(V_y(t))$, which can be expressed as the solution of the SDE~\eqref{AIMDeq0}, with $\alpha{=}\kappa_0$, $\beta{=}\kappa_4/(2\kappa_3 y)$ and $\gamma{=}1/\kappa_5$. Recall that the points $(t_n)$ do not depend on $N$.

It is straightforward, by induction, to extend this result. For any $n$, the relation
\begin{multline*}
  \lim_{N\to+\infty} \E\left(\Phi\left(\left(\frac{X_2^N(\ell_0^N(t))}{N},\frac{X_4^N(\ell_0^N(t))}{\sqrt{N}}\right)\right)\ind{t_n{\le}T{<}t_{n+1}}\right)
\\=  \E\left(\Phi\left(\left(y,V_y(t)\right)\right)\ind{t_n{\le}T{<}t_{n+1}}\right),
\end{multline*}
holds. 

With the martingale problem formulation associated to the SDE~\eqref{AIMDeq0}, it is not difficult to see that $({V}_y(t))$ is a Markov process with infinitesimal generator ${\cal A}_y$. See Section~4.4 of~\citet{KurtzEthier}. The proposition is proved. 
\end{proof}

\subsection{Convergence on the Normal Timescale}\label{NTS}
We can now state  convergence results for the processes $(X_2^N(t)/N,X_4^N(t)/\sqrt{N})$ without a time-change. With the same notations as in the proof of the last proposition, one has to consider, for example, time intervals of the type  $[t_1{+}\tau_N^1, t_1{+}\tau_N^1{+}\tau_N^2)$. By Proposition~\ref{Hconv}, its width $\tau_N^2$ is converging in distribution to $0$, whereas the number of jumps on it is of the order of $\sqrt{N}$. Because of that, the $J_1$-topology is not a convenient topology for a convergence result.

  We will use another Skorohod topology, the $M_1$-topology, which allows the accumulation of jumps in a small neighborhood.  We will also consider a non-Skorohod topology, the $S$-topology. Both topologies have their pros and cons, see~\cite{Jakubowski}. For the $S$-topology, tightness criteria are somewhat simpler and the sum (of processes) is a continuous mapping for this topology, which is not the case for the $M_1$ and $J_1$-topologies. Unfortunately this is not a metrisable topology and the $S$ convergence does not imply the convergence  in distribution of the finite marginals. For both topologies, the integration of processes is a continuous functional, which is a key property to investigate the asymptotic integral equations verified by the possible limiting points.

Chapter~12 of~\citet{Whitt} contains an in-depth  presentation of the $M_1$-Skorohod topology with many details (and other Skorohod topologies) and~\citet{Jakubowski} for  the $S$-topology.  See also~\citet{Kern} for a quick and nice introduction to the intricacies of Skorohod topologies.
  
\begin{proposition}\label{PropM1}
If  $X_N(0){=}(0,y_N,1,v_N)$ satisfy Relation~\eqref{InitIV}, then for the convergence in distribution for the $M_1$-Skorohod topology and the $S$-topology,
  \[
  \lim_{N\to+\infty} \left(\frac{X_2^N(t)}{N},\frac{X_4^N(t)}{\sqrt{N}}\right)=(y,V_y(t)),
  \]
  where $(V_y(t))$ is the Markov process on $\R_+$ defined in Proposition~\ref{PropJ1}. 
\end{proposition}
\begin{proof}
Recall that $(t_n)$ is the non-decreasing sequence of points of ${\cal P}_0((0,\kappa_0],\diff t)$, as in the proof of Proposition~\ref{PropJ1} they can be considered as ``fixed'' since they do not depend on $N$. We proceed as in the proof of Proposition~\ref{PropJ1} by working on the events $\{t_n{\le}T{<}t_{n+1}\}$. As before, it is sufficient to consider $n{=}1$. As before, the Skorohod representation theorem is used, implicitly this time, with a convenient probability space, \ldots

We begin with the $M_1$-topology. 
 Compared to the proof of the last proposition, the same type of arguments are used.  We only have to ``insert'' the time interval $[t_1{+}\tau_N^1, t_1{+}\tau_N^1{+}\tau_N^2)$ which will give the jump at time $t_1$. Recall that the $M_1$-topology is weaker than the $J_1$-topology.

    We introduce the processes $(\widetilde{X}_2^N(t),\widetilde{X}_4^N(t))$ defined by
    \[
    (\widetilde{X}_2^N(t),\widetilde{X}_4^N(t))=    ({X}_2^N(t),{X}_4^N(t)), \quad t{\not\in}(t_1{+}\tau_N^1, t_1{+}\tau_N^1{+}\tau_N^2]
    \]
    and, for $t{\le}T$,
    \[
    \begin{cases}
      \widetilde{X}_2^N(t){=}X_2^N(t_1{+}\tau_N^1) \text{ for } t{\in} (t_1{+}\tau_N^1, t_1{+}\tau_N^1{+}\tau_N^2)\\
      \widetilde{X}_2^N(t_1{+}\tau_N^1{+}\tau_N^2){=}X_2^N(t_1{+}\tau_N^1{+}\tau_N^2)
    \end{cases}
    \]
    and
    \[
    \widetilde{X}_4^N(t)=X_4^N(t_1{+}\tau_N^1){+} \kappa_3 y N(t{-}(t_1{+}\tau_N^1))\text{  for  } t{\in} (t_1{+}\tau_N^1, t_1{+}\tau_N^1{+}\tau_N^2).
    \]
By using the  notations of the last proof,  on the time interval $[t_1{+}\tau_N^1, t_1{+}\tau_N^1{+}\tau_N^2)$, $(\widetilde{X}_2^N(t))$ is kept constant equal to
  $X_2^N(t_1{+}\tau_N^1)$    and $(\widetilde{X}_4^N(t)/\sqrt{N})$ is a linear interpolation between the points
\begin{multline*}
  \frac{X_4^N(t_1{+}\tau_N^1)}{\sqrt{N}}\stackrel{\text{dist}}{\sim}\overline{v}=v\exp({-}\kappa_5 t_1)
\\  \text{ and  } \frac{X_4^N(t_1{+}\tau_N^1)}{\sqrt{N}}{+}\kappa_3 y \sqrt{N}\tau_2^N \stackrel{\text{dist}}{\sim} V_1{=}\overline{v}+\kappa_3 y H_{y, \overline{v}},
\end{multline*}
which correspond to $V_y(t_1{-})$ and $V_y(t_1)$.

By using Corollary~\ref{Corol1}, where the underlying topology is the uniform norm,  and Definition~(3.4) of~\citet{Whitt} for the distance for $M_1$ on ${\cal D}([0,T))$, on the event $\{t_1{\le}T{<}t_{2}\}$, the two processes $(\widetilde{X}_2^N(t)/N,\widetilde{X}_4^N(t)/\sqrt{N})$ and $(X_2^N(t)/N,X_4^N(t)/\sqrt{N})$ are arbitrarily close as $N$ goes to infinity.

  The classical example~(3.1) p.~80 of~\citet{Whitt} and its parametrization~(3.4) show that  on the event $\{t_1{\le}T{<}t_{2}\}$, the sequence of processes
  \[
  \left(\frac{\widetilde{X}_2^N(t)}{N},\frac{\widetilde{X}_4^N(t)}{\sqrt{N}}\right)
  \]
  converges in distribution for the $M_1$-topology to the process $(y,V_y(t))$.

  For the $S$-topology, for $x{\in}{\cal D}((0,T))$ and $\eta{>}0$, $N_{\eta}(x)$ is the number of oscillations of order  $\eta$ for $(x(t))$ on $[0,T]$, i.e. for $k{\ge}1$, $N_{\eta}(x){\ge}k$ holds if there exists $t_1{<}t_2{<}\cdots{<}t_{2k}$, such that $|x(t_{2i}){-}x(t_{2i-1})|{>}\eta$ for $i{\in}\{1,\ldots,k\}$.

  Proposition~(3.1)~(iii) of~\citet{Jakubowski} shows that if, for any $\eta{>}0$, the sequence  of random variables 
  \[
  \left(\left\|\frac{X_2^N}{N}\right\|_T,\left\|\frac{X_4^N}{\sqrt{N}}\right\|_T, \frac{X_2^N(t)}{N}, N_\eta\left(\frac{X_2^N(t)}{N}\right),N_\eta\left(\frac{X_4^N(t)}{\sqrt{N}}\right)\right)
\]
is tight, then the sequence of processes  $(X_2^N(t)/N,X_4^N(t)/\sqrt{N})$ is tight for the $S$-topology. This is seen by the same arguments as in the  proof of Proposition~\ref{PropJ1} for the  time intervals $[0,t_1)$, $[t_1,t_1{+}\tau_1^N)$ and, with Corollary~\ref{Corol1} for the time interval $[t_1{+}\tau_1^N,t_1{+}\tau_1^N{+}\tau_2^N)$.

Since, by the $M_1$-convergence,  the finite marginals of these processes converge in distribution to the corresponding  finite marginals of $(y,V_y(t))$,   Theorem~3.11 of~\citet{Jakubowski} gives the convergence in distribution for the $S$-topology to  $(y,V_y(t))$

The proposition is proved. 

\end{proof}

\section{A Stochastic Averaging Principle}\label{FNII}
\begin{figure}[htpb]
        \begin{subfigure}[b]{0.45\textwidth}
                \centering
                \scalebox{2}{\includegraphics[width=.45\textwidth]{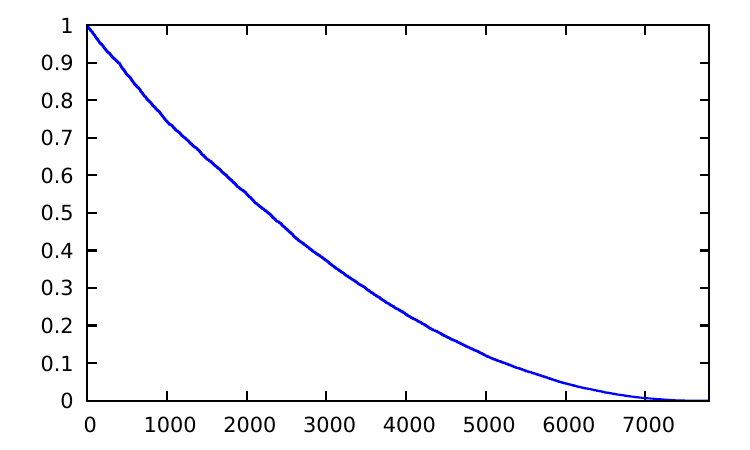}}
                \caption{$(X_2^N(t)/N)$}
        \end{subfigure} 
        ~
        \begin{subfigure}[b]{0.45\textwidth}
                \centering      
                \scalebox{2}{\includegraphics[width=.45\textwidth]{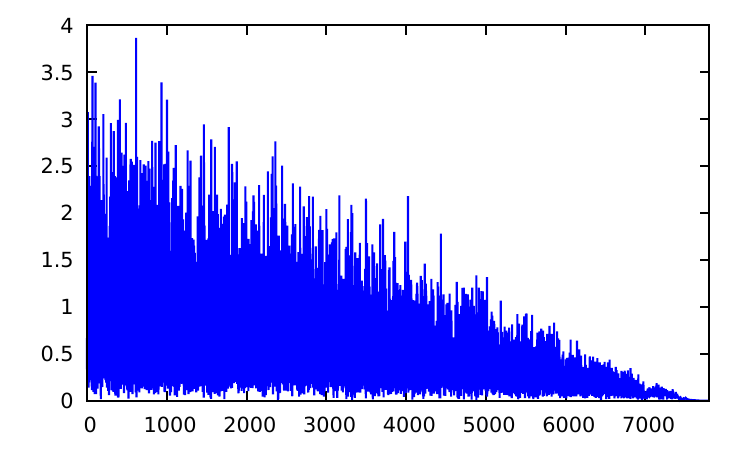}}
                \caption{$(X_4^N(t)/\sqrt{N})$}
        \end{subfigure}
        \caption{CRN with Four Nodes. Full  simulation.\\  Initial state $(0,N,0,0)$, $\kappa_i{=}1$, $i{=}0,$\ldots,$5$ and $N{=}9{\cdot}10^6$.}     
        \label{FigAV}
\end{figure}

Starting from the state $(0,y_N,1,v_N)$ with $(y_N,v_N){\sim}(yN,v\sqrt{N})$, Proposition~\ref{Hconv} shows that, when the value of $(X_3^N(t))$ switches back to $0$,  the process  $(X_4^N(t))$ has increased by an amount of the order of $\sqrt{N}$. This implies that the process  $(X_2^N(t))$ has decreased by an amount also of the order of $\sqrt{N}$ at this instant. Recall that starting from a state for which the third coordinate  is  $0$, it becomes $1$ again after a duration of time stochastically lower bounded by an exponential distribution with parameter $\kappa_0$.  Hence, to have a decay of the order of $N$ for $(X_2^N(t))$, one needs a number of such cycles of the order of $\sqrt{N}$ and, therefore, this suggests that the ``correct'' timescale to observe a decay of $(X_2^N(t))$ is $(\sqrt{N}t)$. This is the main result of this section. 

We introduce $(U_N(t)){=}(U_1^N(t),U_2^N(t),U_3^N(t),U_4^N(t))$, with $(U_1^N(t)){=}(0)$, and the other coordinates are the solution of the SDE
\begin{equation}\label{SDEm42}
\begin{cases}
\diff U_2^N(t)&{=}\, {-}{\cal P}_{3}\left(\left(0,\kappa_{3}U_{2}^NU_{3}^N(t{-})\right),\diff t\right),\\
\diff U_3^N(t)&={\cal P}_{0}\left(\left(0,\kappa_{0}\right),\diff t\right){-}{\cal P}_{4}\left(\left(0,\kappa_{4}U_{3}^NU_{4}^N(t{-})\right),\diff t\right),\\
\diff U_4^N(t)&={\cal P}_3((0,\kappa_3U_2^NU_3^N(t{-})),\diff t){-}{\cal P}_{5}((0,\kappa_{5}U_4^N(t{-})),\diff t),
\end{cases}
\end{equation}
with $(U_N(0)){=}(0,y_N,1,v_N)$ and the sequence $(y_N,v_N)$ satisfies Relation~\eqref{InitIV}.

We first give a heuristic motivation for the introduction of $(U_N(t))$. The difference between  the processes  $(X_N(t))$ and $(U_N(t))$ lies in the fact that when there is a jump of ${\cal P}_{0}\left(\left(0,\kappa_{0}\right),\diff t\right)$, for $U_N$ it is transferred right away to the coordinate $U_3^N$ and for $X_N$ it goes to $X_1^N$ and then to $X_3^N$ at a rate $\kappa_2X_2^N(t)$ at time $t$. If $X_2^N(t)$ is of the order of $N$,  the difference between $(X_N(t))$ and $(U_N(t))$ holds on a time interval whose duration is of the order of $1/N$ and the probability that there is another event that will change the coordinates during that time is of the order of  $1/N$. Hence on a time interval $[0,\sqrt{N}T]$, given that the number of jumps of  ${\cal P}_{0}$ is of the order of $\sqrt{N}$, the two processes should be ``close''. Proposition~\ref{U=Z} at the end of this section gives a formal assessment of this non-rigorous picture. The limit results for $(X_N(t))$ will be a consequence of the limit results obtained for $(U_N(t))$. 

From now on, we investigate the asymptotic behavior of $(U_N(t))$.  The general strategy is of first considering, via  a time change, the time evolution of the process when the third coordinate is above $1$
\subsection{Time Evolution when ${\mathbf U_3^N}$ is non-zero}
We introduce the local time of the excursions of the process $(U_3^N(t))$ above $1$. 
\begin{definition}[Time Change]\label{elldef}
For $t{\ge}0$,
\[
L_1^N(t)\steq{def} \int_0^t \ind{U_3^N(s){\ge}1}\diff s\quad \text{and}\quad \ell_1^N(t)\steq{def}\inf\{s{>}0: L_1^N(s){>} t\}. 
\]
\end{definition}
Before introducing the formal time change arguments of this section, to motivate the description of the process $(U_N(t))$ on the timescale $(\ell_1^N(t))$,  we give a quick presentation of the dynamics involved. 

In view of the SDEs~\eqref{SDEm42}, if $U_3^N$ is $1$ at time $t_0$,  it becomes $0$ at rate $\kappa_4U_4^N(t_0{-})$ and stays at $0$ for an exponentially  distributed amount of time with parameter $\kappa_0$.

When $(U_3^N(t))$ is $0$ on a time interval $[t_0,t_0{+}a)$, for some $a{>}0$,  the process $(U_2^N(t))$ does not change and the process $(U_4^N(t))$ can only decrease. At time $t_0$,  the $p{=}U_4^N(t_0)$  elements of $(U_4^N(t))$ can be seen as  having lifetimes $E_{\kappa_{5},1}$, \ldots, $E_{\kappa_{5},p}$. These variables are exponentially distributed with parameter $\kappa_5$.  At time $t_0{+}a$ there remain only those whose  lifetime is greater than $a$. 

  This suggests quite naturally  the following description of our system. Let ${\cal N}$ be a  Poisson marked point process  on $\R_+^2{\times}\R_+{\times}\R_+^\N$ with intensity measure $$\diff s{\otimes}\diff t{\otimes}\kappa_0\exp({-}\kappa_0a)\diff a{\otimes}Q(\diff b),$$ where $Q$ is the distribution of $(E_{\kappa_5,i})$ on $\R_+^\N$.

  The process $(Z_N(t)){=}(Z^N_2(t),Z^N_3(t),Z^N_4(t))$ is the solution of the SDE
\begin{equation}\label{SDEZ}
\begin{cases}
\diff Z_2^N(t){=}\, {-}{\cal P}_{3}\left(\left(0,\kappa_{3}Z_{2}^NZ_{3}^N(t{-})\right),\diff t\right),\\
\diff Z_3^N(t)={\cal P}_{0}\left(\left(0,\kappa_{0}\right),\diff t\right){-}\ind{Z_3^N(t{-})\geq 2}{\cal P}_{4}\left(\left(0,\kappa_{4}Z_{3}^NZ_{4}^N(t{-})\right),\diff t\right),\\
\diff Z_4^N(t)={\cal P}_3((0,\kappa_3Z_2^NZ_3^N(t{-})),\diff t){-}{\cal P}_{5}((0,\kappa_{5}Z_4^N(t{-})),\diff t),\\
 \qquad{-}\ind{Z_3^N(t{-}){=}1}\displaystyle\int_{a,b}S(Z_4^N(t{-}),a,b) {\cal N}\left(\left(0,\kappa_4Z_4^N(t{-})\right],\diff t,\diff a,\diff b\right),
\end{cases}
\end{equation}
where $S({\cdot})$ is defined by Relation~\eqref{Sdef}. It turns out that $(Z_N(t))$ has the same distribution as the  time changed process $(U_N(t))$.
\begin{proposition}\label{TC1Prop}
  For $N{\geq}1$, if $(U_N(0)){=}(Z_N(0)){=}(y_N,1,v_N)$ then
  \[
  \left(Z_N(t)\right)\steq{dist} \left(U_N(\ell_1^N(t))\right),
  \]
where $(U_N(t))$, $(Z_N(t))$, and $(\ell_1(t))$ are respectively defined by Relations~\eqref{SDEm42}, \eqref{SDEZ} and Definition~\eqref{elldef}
\end{proposition}
\begin{proof}
  We give a sketch of the proof. It is essentially a consequence of a repeated use of the strong Markov property of Poisson processes.
If $U_N(0){=}(y_N,1,v_N)$, let $\tau_U^N$ be the first instant when $(U_3^N(t))$ hits $0$, i.e. the first jump of the counting process
  \[
\left(\int_0^t\ind{U_3^N(s){=}1}{\cal P}_{4}\left(\left(0,\kappa_{4}U_{4}^N(s{-})\right),\diff s\right)\right). 
  \]

On the time interval $[0,\tau_U^N)$ $(U_N(t))$ satisfies the SDEs~\eqref{SDEZ}.
If $E_{\kappa_0}$ is the first point of $({\cal P}_{0}(\left(0,\kappa_{0}),(\tau_U^N,\tau_U^N{+}t]\right)$, $E_{\kappa_0}$ is an exponential random variable with parameter $\kappa_0$ independent  of ${\cal F}_{\tau_U^N}$, and on the time interval $[\tau_U^N, \tau_U^N{+}E_{\kappa_0})$, the coordinate $(U_4^N(\tau_U^N{+}t))$ satisfies the SDE,
    \[
    \diff A_N(t)= {-}{\cal P}_{5}((0,\kappa_{5}A_N(t{-})),\tau_U^N{+}\diff t),
    \]
with $A_N(0){=}U_4^N(\tau_U^N)$. It is easily seen that the process $(A_N(t))$ has the same distribution as
 \[
\left(\sum_{i=1}^{U_4^N(\tau_U^N)}\ind{E_{\kappa_5,i}> t}\right),
\]
where $(E_{\kappa_5,i})$ are i.i.d. exponential random variables with parameter $\kappa_5$ independent of ${\cal F}_{\tau_U^N}$. 
We have therefore that
\begin{multline*}
U_4(\ell_1^N(\tau_U^N)){-} U_4(\ell_1^N(\tau_U^N{-}))=U_4(\tau_U^N{+}E_{\kappa_0}){-}U_4(\tau_U^N)\\=\sum_{i=1}^{U_4^N(\tau_U^N)}\ind{E_{\kappa_5,i}> E_{\kappa_0}}-U_4^N(\tau_U^N)=-S(U_4^N(\tau_U^N),E_{\kappa_0},(E_{\kappa_5,i})),
\end{multline*}
where $S$ is the function defined by Relation~\eqref{Sdef}. One can proceed by induction on the successive instants of return to $1$ from $0$ of $(U_3^N(t))$. The proposition is proved. 
\end{proof}
Note that by considering the process $(Z_N(t))$, we remove  the time intervals where $(U_2^N(t)/N)$ is constant and therefor e no effect on its decay. The relevant timescale to see the decrease of $(Z_2^N(t)/N)$ is now the normal timescale $(t)$, and we will see later  that $\ell^N_1(t){=}O(\sqrt{N})$. 

\begin{definition}\label{defmN}
 The occupation measure $m_N$ of $(Z_4^N(t)/\sqrt{N})$ is
  \[
  \croc{m_N,F}=\int_0^{+\infty} F\left(s,\frac{Z_4^N(s)}{\sqrt{N}}\right)\diff s,
  \]
  for $F{\in}{\cal C}_c(\R_+^2)$.
\end{definition}

We can now state a key result of this section.
\begin{theorem}\label{TheoCVZ}
If $(Z_N(t))$, the solution of the SDEs~\eqref{SDEZ} and $\mu_Z^N$ the occupation measure of Definition~\ref{defmN} are such that $(Z_N(0)){=}(y_N,1,v_N)$ and $(y_N,v_N)$ satisfies Relation~\eqref{InitIV}, then, for the convergence in distribution,
  \[
  \lim_{N\to+\infty} ((Z_2^N(t)),\mu_Z^N)= \left((z_2(t)),\mu_Z^\infty\right),
  \]
with $(z_2(t)){=}(ye^{-\kappa_3 t})$,  and, for  any $F{\in}{\cal C}_c(\R_+^2)$,
    \[
\int_{\R_+^2} F\left(s,x\right)\mu_Z^\infty(\diff s,\diff x) =\int_{\R_+^2} F(s,\sqrt{x})\Gamma_0\left(\frac{\kappa_0}{2\kappa_5}{+}\frac{1}{2}, \frac{\kappa_4}{2\kappa_3 z_2(s)}\right)(\diff x)\diff s,
  \]
where  $\Gamma_0({\cdot},{\cdot})$ is the distribution of Definition~\ref{GamDef}.
\end{theorem}
The proof of this theorem is carried out dividing the time integral according to the value of $Z_3^N(t)$, recall that it is always above $1$.  
\begin{enumerate}
\item  In the time intervals where $Z_3^N(\cdot){=}(1)$, we will show that the limit result can be deduced from Theorem~\ref{LimitOccApp} of Section~\ref{PartSec}. 
\item For the time intervals during which $Z_3^N(t){\geq}2$, the goal will be of showing that these intervals do not contribute to the final limit,  more precisely  that, for the convergence in distribution, 
	\[
 \lim_{N\to +\infty}\left(\int_0^{+\infty} \ind{Z_3^N(s)\geq 2}F\left(s,\frac{Z_4^N(s)}{\sqrt{N}}\right)\diff s\right)=0,
 \]
 for some convenient class of functions on $\R_+^2$. This is done in the following way.
 \medskip
 
If $t_1^3$ is such that $Z_3^N(t_1^3){=}2$, we define $\tau_Z^N{=}\inf\{t\geq 0: Z_3^N(t_3^N{+}t){=}1\}$.  We will show that $\tau_Z^N=O(1/\sqrt{N})$,  
\begin{itemize}
	\item first, we show  after a time of the order of $O(1/\sqrt{N})$, the process $(Z_4^N(t_1^3{+}t))$ reaches  a value the of order $\sqrt{N}$, and stays at this order of magnitude.
	\item secondly, since $Z_3^N(t)$ decreases at a rate $\kappa_4Z_4^N(t)$, we show that $Z_3^N(t)$ goes back to $1$ after a time of the order of $O(1/\sqrt{N})$. 
\end{itemize}
And to conclude, we only have to notice that these time intervals $(t_1^3, t_1^3+\tau_Z^N)$ are in a finite number in $[0, T]$ for any $T\geq 0$, since every time interval is separated by an exponentially distributed variable with parameter $\kappa_0$.
\end{enumerate}
The proof will be done by induction, on an ``excursion'' of $Z_3^N(t)$ over $2$. On a first time interval $[t_1^3, t_1^3{+}\tau_Z^N)$, and then iterating. 

\begin{proof}
If $(Z_N(0)){=}(y_N,1,v_N)$,  $t_1^3$, the first instant of jump ${+}1$ of $(Z_3^{N}(t))$ has an exponential distribution with parameter $\kappa_0$. Up to time $t_1^3$ the process $(Z_2^N(t),Z_4^N(t))$ has the same distribution as the process  $(A_N(t),B_N(t))$, the solution of the SDE,
\begin{equation}\label{mSDE}
\begin{cases}
\diff A_N(t)= {-}{\cal P}_{3}\left(\left(0,\kappa_{3}A_N(t{-})\right),\diff t\right),\\
\diff B_N(t)={\cal P}_3((0,\kappa_3A_N(t{-})),\diff t){-}{\cal P}_{5}((0,\kappa_{5}B_N(t{-})),\diff t),\\
 \qquad{-}\displaystyle\int_{a,b}S(B_N(t{-}),a,b) {\cal N}\left(\left(0,\kappa_4B_N(t{-})\right],\diff t,\diff a,\diff b\right).
\end{cases}
\end{equation}
with $(A_N(0),B_N(0)){=}(y_N,v_N)$. 

It is straightforward to show that $(A_N(t)/N)$ is converging in distribution to $(y\exp(-\kappa_3 t))$ and that the asymptotic behavior of the occupation measure associated to $(B_N^0(t)/\sqrt{N})$ can be obtained from Theorem~\ref{LimitOccApp} of Section~\ref{PartSec}. It is not difficult to see that the additional term due to ${\cal P}_{5}$ in the SDE defining $(Z_4^N(t))$ does not play a role for this limit result.

If $Z_3^{N}(t_1^3){=}2$, let  $\tau_Z^N$ , be  the hitting time of $1$ of the process $(Z_3^{N}(t{+}t_1^3))$.  
In view of the first equation of SDE~\eqref{SDEZ},  with high probability, we have that, for any $t{\in}[0,T]$, $Z_2^N(t){\ge}\delta N$. To simplify our arguments, since we are dealing with convergence in distribution,  we assume that this relation holds almost surely, it is not difficult to modify our proof accordingly, at the expense of additional terms. If $(C_N(t),D_N(t))$ is the solution of the SDE,
\begin{align*}
\diff C_N(t)&=\widetilde{\cal P}_{0}\left(\left(0,\kappa_{0}\right),\diff t\right){-}\ind{C_N(t{-}){>}0}\widetilde{\cal P}_{4}\left(\left(0,\kappa_{4}(1{+}C_N(t{-}))D_N(t{-})\right),\diff t\right),\\
\diff D_N(t)&=\widetilde{\cal P}_3((0,2\kappa_3\delta N),\diff t){-}\widetilde{\cal P}_{5}((0,\kappa_{5}D_N(t{-})),\diff t),
\end{align*}
with $C_N(0){=}1$ and $D_N(0){=}0$, and where, for $i{\in}\{0,3,4,5\}$, $\widetilde{\cal P}_i(\diff s,\diff t)$ is the Poisson process ${\cal P}_i$ shifted at $t_1^3$, i.e. ${\cal P}_i(\diff s ,t_1^3{+}\diff t)$.
A simple coupling can be constructed so that the relations
$D_N(t){\le}Z_4^N(t_1^3{+}t)$ and $Z_3^N(t){\le}C_N(t_1^3{+}t){+}1$ hold for all $0{\le}t{\le}\tau_Z^N$.

Standard arguments of stochastic calculus, give that, for the convergence in distribution, the relation
\[
\lim_{N\to+\infty}\left(\frac{D_N(t/\sqrt{N})}{\sqrt{N}}\right)=(2\kappa_3\delta t)
\]
holds and if $\tau_D^N$ is the hitting time of $\lceil \kappa_3\delta \sqrt{N}\rceil$, then
\[
\limsup_{N\to+\infty} \sqrt{N}\E(\tau_D^N) <{+}\infty, 
\]
and, since $\left(C_N(\tau_D^N){-}1\right)^+$ is bounded by the number of new arrivals for $(C_N(t))$ on the time interval $[0,\tau_D^N)$
  \[
\limsup_{N\to+\infty} \sqrt{N}\E\left(\left(C_N(\tau_D^N){-}1\right)^+\right) <{+}\infty. 
\]
For $T{>}0$, starting from $\tau_D^N$, the process $(D_N(t))$ stays above $\kappa_3\delta \sqrt{N}$ with high probability on a time interval $[\tau_D^N,\tau_D^N{+}T/\sqrt{N}]$. With the same argument as before, we assume that this relation holds almost surely. 

If $\tau_C^N$ is the hitting time of $0$ by $(C_N(\tau_D^N{+}t))$, then $\tau_Z^N{\le}\tau_D^N{+}\tau_C^N$. The integration of the SDE for $(C_N(t))$ gives the relation
\begin{multline*}
\E\left(C_N(\tau_D^N)\right){+}\kappa_0\E(\tau_C^N{\wedge}t){-}\kappa_{4}\kappa_3\delta \sqrt{N} \E\left(\int_{0}^{\tau_C^N{\wedge}t} (1{+}C_N(\tau_D^N{+}s))\diff s\right)\\=\E\left(C_N(\tau_D^N{+}\tau_C^N{\wedge}t)\right),
\end{multline*}
hence, since $C_N(t){\ge}1$ on the time interval $[\tau_D^N,\tau_D^N{+}\tau_C^N)$, we obtain
  \[
\left(\kappa_{4}\kappa_3\delta \sqrt{N}{-}\kappa_0\right) \E\left(\tau_C^N{\wedge}t\right)\le \E\left(C_N(\tau_D^N)\right)\le  1{+} \E\left(\left(C_N(\tau_D^N){-}1\right)^+\right),
\]
hence
\[
\limsup_{N\to+\infty} \sqrt{N}\E(\tau_C^N) <{+}\infty.
\]
By gathering these results, and since, for $t{\ge}0$,
\[
Z_4^N(t_1^3{+}t/\sqrt{N})-Z_4^N(t_1^3){\le}\widetilde{\cal P}_3((0,2\kappa_3\delta N),[0,t/\sqrt{N}]),
\]
we have finally obtained that
\[
\limsup_{N\to+\infty} \sqrt{N}\E(\tau_Z^N) <{+}\infty,\quad \text{and}\quad
\limsup_{N\to+\infty} \frac{\E(Z_4^N(t_1^3+\tau_Z^N)-Z_4^N(t_1^3))}{\sqrt{N}}<{+}\infty. 
\]
This shows that, if $f$ is a bounded Borelian function on $\R_+^2$, the sequences of random variables
\[
\left(\int_{t_1^3}^{t_1^3{+}\tau_Z^N}\ind{Z_3^N(s){\ge}2}f\left(\frac{Z_2^N(s)}{N},\frac{Z_4^N(s)}{\sqrt{N}}\right)\diff s\right) \text{ and }
\left(\left|\frac{Z_2^N(t_1^3{+}\tau_Z^N){-}Z_2^N(t_1^3)}{N}\right|\right)
\]
converge in distribution to $0$. We can now apply the convergence result, Theorem~\ref{LimitOccApp}, for the processes $(Z_2^N(t),Z_4^N(t))$ starting from time $t_1^3{+}\tau_Z^N$, as if it was starting from time $t_1^3$ with $Z_3^N(t_1^3){=}1$, since their first coordinate is of the same order of magnitude in $N$, and that the limit result for the occupation measure does not depend on the initial value of $Z_4^N(t_1^3{+}\tau_Z^N)$ as long as it is of the order of $N$ at most. 

We conclude the proof of the proposition by induction on the successive jumps $(t_i^3)$ of ${\cal P}_0((0,\kappa_0),\diff t)$ on the time interval $[0,T]$.
\end{proof}
Note that the induction holds because the averaging result of Theorem~\ref{LimitOccApp} do not need for $Z_4^N(t)$ to start from a state of the order of $\sqrt{N}$, but only  $O(N)$. This particular property is vital here, since we do not have the control of the process $(Z_4^N(t)/\sqrt{N})$ over some time interval, but only of its time integrals. 

We have seen in the proof that the process $(Z_3^N(t))$ is identically $1$, \emph{as long as we consider time integrals}. The following Corollary is easily deduced from Relations~\eqref{Fq2} and~\eqref{mSDE}.  

\begin{corollary}\label{CorZ}
Let $F{\in}{\cal C}(\R_+^2)$ with support on $[0,T]{\times}\R_+$ such that there exist  $c_1$, $c_2{>}0$ such that for any $s,x\in \R_+$, 
$F(s,x)\leq c_1{+}c_2x$ holds for $(s,x){\in}[0,T]{\times}\R_+$, then under the assumptions of Theorem~\ref{TheoCVZ}, the sequence
  \[
 \left(\int_0^{+\infty} \ind{Z_3^N(s)=1}F\left(s,\frac{Z_4^N(s)}{\sqrt{N}}\right)\diff s\right) 
 \]
converges in distribution to $\croc{\mu_Z^\infty,F}$.
\end{corollary}
\subsection{Time Evolution on the Timescale $(\sqrt{N}t)$}
We are now going to express $(U_N(t))$ in terms of $(Z_N(t))$. Recall that the process $(Z_2^N(t),Z_4^N(t))$ analyzed in the last section is just the process $(U_2^N(t),U_4^N(t))$ with the time intervals during which $(U_3^N(t))$ is $0$ removed. See Proposition~\ref{TC1Prop}.

We denote by $(Z_N(t))$ the solution of the SDE~\eqref{SDEZ} starting from $(y_N,1,v_N)$.
 We define
\begin{equation}\label{Hdef}
\left(H_N(t)\right){=}\left(t{+}\int_{(0,t]\times\R_+{\times}\R_+^\N}a\ind{Z_3^N(s-){=}1}{\cal N}\left(\left(0,\kappa_4Z_4^N(s{-})\right],\diff s,\diff a,\diff b\right)\right),
\end{equation}
and the (potential) hitting time of $0$ for $(Z_3^N(t))$ is defined as  $\tau_N$, i.e.
    \[
  \tau_N=\inf\left\{t:\int_{(0,t]\times\R_+{\times}\R_+^\N}\ind{Z_3^N(s-){=}1}{\cal N}((0,\kappa_4Z_4^N(s{-})],\diff s,\diff a,\diff b){\ne}0\right\}.
      \]
Strictly speaking, this is an incorrect presentation for $\tau_N$ since  $(Z_3^N(t))$ never visits $0$. This is in fact meant for $(U_3^N(t))$, as long as the two  processes coincide. 
Additionally,  $(a_N,b_N){\in}\R_+{\times}\R_+^\N$ is  the mark associated to $\tau_N$, i.e. $${\cal N}((0,\kappa_4Z_4^N(\tau_N{-})], \{\tau_N\},\diff a,\diff b)=\delta_{(a_N,b_N)},$$
  $a_N$ and $b_N{=}(b_{N,i})$ are independent and independent of ${\cal F}_{\tau_n}$, with respective distributions, an exponential law with parameter $\kappa_0$ and the law of an i.i.d. sequence of exponential random variables with parameter $\kappa_5$. We define $t_N{=}\tau_N{+}a_N$.
  
We now construct  a process $(\widetilde{U}_N(t)){=}(\widetilde{U}_2^N(t),\widetilde{U}_3^N(t),\widetilde{U}_4^N(t))$ with initial state  $(y_N,1,v_N)$ and 
\begin{equation}\label{SDETU}
  \begin{cases}
 t{<}\tau_N, & {\scriptstyle (\widetilde{U}_2^N,\widetilde{U}_3^N,\widetilde{U}_4^N)(t)=(Z_2^N,Z_3^N,Z_4^N)(t),} \\ \ \\
 \tau_N{\le} t{<}t_N,
& \hspace{-3mm} \begin{cases}
  {\scriptstyle (\widetilde{U}_2^N,\widetilde{U}_3^N)(t){=}(Z_2^N(\tau_N{-}),0),} \\
 {\scriptstyle   \widetilde{U}_4^N(t){=}\hspace{-1mm}\sum_{i=1}^{Z_4^N(\tau_N{-})}\ind{b_{N,i}{>}t{-}\tau_N}{=}Z_4^N(\tau_N{-}){-}S\left(Z_4^N(\tau_N{-}),t{-}\tau_N,b_N\right)},
\end{cases}\\ \ \\
t{=}t_N,&  \hspace{-5mm}{\scriptstyle  (\widetilde{U}_2^N,\widetilde{U}_3^N,\widetilde{U}_4^N)(t_N){=}\left(\widetilde{U}_2^N(t_N{-}),1,Z_4^N(\tau_N{-}){-}S\left(Z_4^N(\tau_N{-}),a_N,b_N\right)\right)}.
  \end{cases}
\end{equation}
We have constructed the process $(\widetilde{U}_N(t))$  between two visits of the third coordinate to $1$. We can construct by induction the process on the whole real half-line.
  \begin{proposition}
    For $N{\geq}1$, the processes $(U_N(t))$ and $(\widetilde{U}_N(t))$ defined respectively  by Relations~\eqref{SDEm42} and~\eqref{SDETU} have the same distribution.
Furthermore the relations
    \[
\left(\int_0^{H_N(t)}\ind{\widetilde{U}_3^N(s)\ge 1}\diff s\right)=(t) \text{ and } \left(\widetilde{U}_2^N\left(H_N(t)\right)\right)=\left(Z_2^N\left(t\right)\right)
\]
hold. 
  \end{proposition}
We have in particular the identity $(H_N(t)){\steq{dist}}(\ell_1^N(t))$. See Definition~\ref{elldef}.
  \begin{proof}
    The proof of the identity in distribution is analogous to the proof of Proposition~\ref{TC1Prop}. It relies again on strong Markov properties of Poisson processes and the representation of the process $(A_N(t))$ used in this proof.

The first relation comes directly from the construction of $(\widetilde{U}_N(t))$. The last relation is a consequence of the fact that $(\widetilde{U}_2^N(t))$ does not change on the times intervals where $(\widetilde{U}^N_3(t))$ is null. 
  \end{proof}
\begin{proposition}\label{PropH}
If $(Z_N(t))$ is the solution of the SDEs~\eqref{SDEZ} with the initial condition $(y_N,1,v_N)$ and $(y_N,v_N)$ satisfying Relation~\eqref{InitIV} then,  for the convergence in distribution, 
  \[
  \lim_{N\to+\infty} \left(\frac{H_N(t)}{\sqrt{N}}\right)=\left(t_\infty\left(1{-}e^{-\kappa_3 t/2}\right)\right),
  \]
with
\begin{equation}\label{cH}
t_\infty\steq{def}\sqrt{y}\frac{\sqrt{2\kappa_4}}{\kappa_5\sqrt{\kappa_3}}\frac{\Gamma({\kappa_0}/{(2\kappa_5)})}{\Gamma({\kappa_0}/{(2\kappa_5)}{+}{1}/{2})}.
\end{equation}
\end{proposition}
\begin{proof}
For $t{\ge}0$,
\[
\frac{H_N(t)}{\sqrt{N}} =\frac{t}{\sqrt{N}}{+}M_N(t){+} \frac{\kappa_4}{\kappa_0}\int_0^t \ind{Z_3^N(s-){=}1}\frac{Z_4^N(s)}{\sqrt{N}}\diff s,
\]
where $(M_N(t))$ is a martingale whose previsible increasing process is
\[
\left(\croc{M_N}(t)\right)=\left(2\frac{\kappa_4}{\kappa_0}\int_0^t\ind{Z_3^N(s-){=}1}\frac{Z_4^N(s)}{N}\diff s\right).
\]
Theorem~\ref{TheoCVZ} and Doob's Inequality give that $(M_N(t))$ is converging in distribution to $0$ and, with Corollary~\ref{CorZ},  that $(H_N(t)/\sqrt{N})$  converges to
\begin{multline*}
\left(\frac{\kappa_4}{\kappa_0}\int_0^t\int_0^{+\infty} \sqrt{x}\Gamma\left(\frac{\kappa_0}{2\kappa_5}{+}\frac{1}{2}, \frac{\kappa_4}{2\kappa_3 z_2(s)}\right)(\diff x) \diff s\right)
\\=\left(\frac{\kappa_4}{\kappa_0}\frac{\Gamma({\kappa_0}/{(2\kappa_5)}{+}{1})}{\Gamma({(\kappa_0)}/{(2\kappa_5)}{+}{1}/{2})}\sqrt{2}\sqrt{\frac{\kappa_3}{\kappa_4}}\int_0^t \sqrt{z_2(s)}\diff s\right),
\end{multline*}
by Relation~\eqref{MomGam}, with $(z_2(t)){=}(y\exp({-}\kappa_3 t))$.  The proposition is proved. 
\end{proof}
\begin{theorem}\label{AP2}
  If $(U_N(t))$ is the solution of the SDEs~\eqref{SDEm42} with the initial condition $(y_N,1,v_N)$ and $(y_N,v_N)$ satisfying Relation~\eqref{InitIV} then,  for the convergence in distribution,
  \[
  \lim_{N\to+\infty}\left(\frac{U_2^N\left(\sqrt{N}t\right)}{N}, t{<}t_\infty\right)=(u_2(t))\steq{def}\left(y\left(1{-}\frac{t}{t_\infty}\right)^2, t{<}t_\infty\right)
  \]
with $t_\infty$ is defined in Proposition~\ref{PropH}. 
\end{theorem}
The quadratic decay can be seen in the simulations of Figure~\ref{FigAV}~(A).
\begin{proof}
For $0{\le}t{<}t_\infty$, define
  \[
  s_N(t)=\inf\{u{>}0:H_N(u)>\sqrt{N} t\},
  \]
Proposition~\ref{PropH}  gives the convergence in distribution
\begin{equation}\label{SNCV}
\lim_{N\to+\infty} \left(s_N(t)\right)= \left({-}\frac{2}{\kappa_3}\ln\left(1{-}\frac{t}{t_\infty}\right)\right).
\end{equation}
We have the identities
  \[
\left(\frac{\widetilde{U}_2^N(\sqrt{N}t)}{N}\right)=  \left(\frac{\widetilde{U}_2^N(H_N(s_N(t)))}{N}\right)=\left(\frac{Z_2^N((s_N(t)))}{N}\right),
\]
the first one is due to the fact that  $\widetilde{U}_2^N$ does not change just after a jump of $(H_N(t))$ and the second holds by Proposition~\ref{PropH}.

We conclude the proof of the theorem with the convergence in distribution of Theorem~\ref{TheoCVZ} and Theorem~3.3 of~\citet{KurtzEthier}. 
\end{proof}

\begin{definition}\label{defmUN}
The occupation measure $\mu_U^N$ of 
  \[
\left(\frac{U_4^N(\sqrt{N}t)}{\sqrt{N}}\right) \text{ is }  \croc{\mu_U^N,F}\steq{def}\int_0^{+\infty} F\left(s,\frac{U_4^N(s\sqrt{N})}{\sqrt{N}}\right)\diff s,
  \]
  for $F{\in}{\cal C}_c(\R_+^2)$.
\end{definition}

\begin{theorem}
If $(U_N(t))$ is the solution of the SDEs~\eqref{SDEm42} with the initial condition $(y_N,1,v_N)$ and $(y_N,v_N)$ satisfying Relation~\eqref{InitIV}, then the sequence $(\mu_U^N)$ converges in distribution to $\mu_U^\infty$ defined by, for  any $F{\in}{\cal C}_c([0,t_\infty){\times}\R_+)$,
  \[
  \int F(s,x)\mu_U^\infty(\diff s, \diff x)=  \int F(s,\sqrt{x})\Gamma_0\left(\frac{\kappa_0}{2\kappa_5}, \frac{\kappa_4}{2\kappa_3y(1-t/t_\infty)^2} \right)(\diff x)\diff s,
  \] 
	where $t_\infty$ is defined by Relation~\eqref{cH} and $\Gamma_0$ is the distribution of Definition~\ref{GamDef}.
\end{theorem}
By using  Proposition~\ref{VoInv} of Section~\ref{V0sec}, note that,   for $0{\le}t{<}t_\infty$,
\[
\Gamma_0\left(\frac{\kappa_0}{2\kappa_5}, \frac{\kappa_4}{2\kappa_3y(1-t/t_\infty)^2} \right)
\]
is the invariant distribution of the infinitesimal generator ${\cal A}_{u_2(t)}$ of Proposition~\ref{PropJ1}, where $(u_2(t))$ is defined by Theorem~\ref{AP2}.
\begin{proof}
  Let $f{\in}{\cal C}_c([0,t_\infty){\times}\R_+)$, we have
  \[
		\croc{\mu_U^N,F}=\frac{1}{\sqrt{N}}\int_0^{+\infty} f\left(\frac{s}{\sqrt{N}},\frac{U_4^N(s)}{\sqrt{N}}\right)\diff s. 
  \]
For $F{\in}{\cal C}_c(\R_+)$ and $S{\geq}0$,  by using the definition of $(\widetilde{U}_4^N(s))$, the relation
\begin{multline}\label{OCeq}
\frac{1}{\sqrt{N}}\int_0^{H_N(S)} F\left(\frac{\widetilde{U}_4^N(s)}{\sqrt{N}}\right)\diff s
=\frac{1}{\sqrt{N}}\int_0^{S} F\left(\frac{Z_4^N(s)}{\sqrt{N}}\right)\diff s\\
+\frac{1}{\sqrt{N}}\int_0^{S} \ind{Z_3^N(s{-})=1}\int_{(a,b)\in \R_+\times \R_+^\N}\left(\int_0^a F\left(\frac{Z_4^N(s{-}){-}S(Z_4^N(s{-}),u,b)}{\sqrt{N}}\right)\diff u\right) \\ \times{\cal N}((0,\kappa_4Z_4^N(s{-})],\diff s,\diff a,\diff b),
\end{multline}
holds on the event $\{H_N(S){<}\sqrt{N}t_\infty\}$ whose probability is converging to $1$ as $N$ gets large. 

\bigskip
\noindent
{\sc Tightness.}\\
For $\eps{>}0$ and  $0{\le}t_0{<}t_\infty$, with Proposition~\ref{PropH}, we get that there exists $S{>}0$ such that, for $N$ sufficiently large
\[
 \P\left(\frac{H_N(S)}{\sqrt{N}}\not\in[t_0,t_\infty)\right)\le \eps,
\]
hence
\begin{multline*}
\E\left(\croc{\mu_U^N,[0,t_0]{\times}[K,{+}\infty]}\right)=
\frac{1}{\sqrt{N}}\int_0^{\sqrt{N}t_0}\P\left(U_4^N(s)\ge K{\sqrt{N}}\right)\diff s\\
\leq \eps t_0{+}\frac{1}{\sqrt{N}}\E\left(\int_0^{H_N(S)}\ind{{U_4^N(s)}\ge K{\sqrt{N}}}\diff s\right). 
\end{multline*}
With Relation~\eqref{OCeq}, we obtain that
\begin{multline*}
\frac{1}{\sqrt{N}} \E\left(\int_0^{H_N(S)}\ind{{U_4^N(s)}\ge K{\sqrt{N}}}\diff s\right)
  \le \frac{1}{\sqrt{N}} \int_0^{S} \P\left(\frac{Z_4^N(s)}{\sqrt{N}}{\ge}K\right)\diff s
 \\ \qquad+\kappa_4\int_0^{S}\E\left(\ind{Z_3^N(s{-})=1}\frac{Z_4^N(s)}{\sqrt{N}}\ind{{Z_4^N(s)}{\ge}K{\sqrt{N}}}\right)\diff s\\
 \le \frac{S}{\sqrt{N}}+\frac{1}{K}\int_0^S\E\left(\ind{Z_3^N(s{-})=1}\left(\frac{Z_4^N(s)}{\sqrt{N}}\right)^2\right)\diff s
\end{multline*}
holds.

Relations~\eqref{Fq2} and~\eqref{mSDE} show that
\[
	\limsup_{N\to +\infty} \E\left(\int_0^S\ind{Z_3^N(s{-})=1}\left(\frac{Z_4^N(s)}{\sqrt{N}}\right)^2\diff s\right)<+\infty.
\] 
Therefore, one can choose $K$ sufficiently large so that the quantity $$\E\left(\croc{\mu_U^N,[0,t_0]{\times}[K,{+}\infty]}\right)$$ is  arbitrarily small for $N$ sufficiently large.  Lemma~1.3 of~\citet{Kurtz1992} gives the tightness of the sequence  $(\mu_U^N)$ of random measures on $[0,t_\infty]{\times}\R_+$.

\bigskip
\noindent
{\sc Identification of the Limit.}\\

With the tightness property and since any limiting point can be represented as in Relation~\eqref{muI}, 
it is enough to identify the limit of
\[
\frac{1}{\sqrt{N}}\int_0^{\sqrt{N}t} F\left(\frac{\widetilde{U}_4^N(s)}{\sqrt{N}}\right)\diff s,
\]
for any $t{\in}[0,t_\infty)$. Since, by Proposition~\ref{PropH}, the process $(H_N(t)/\sqrt{N})$ converges in distribution to a deterministic function, one has to obtain the limit of the sequence
\[
\left(\frac{1}{\sqrt{N}}\int_0^{H_N(S)} F\left(\frac{\widetilde{U}_4^N(s)}{\sqrt{N}}\right)\diff s\right). 
\]
The first term of the right-hand side  of Relation~\eqref{OCeq} converges clearly in distribution to $0$.  The second term  can be written as $J_N(S){+}M_N(S)$,
where, for $t{\ge}0$, 
\begin{multline*}
J_N(t){=}\kappa_4\int_0^{t}\ind{Z_3^N(s=1} \\\times\E\left.\left(\int_0^{E_{\kappa_0}} F\left(\frac{Z_4^N(s)-S(Z_4^N(s),u,(E_{\kappa_5,i}))}{\sqrt{N}}\right)\diff u\right|{\cal F}_s\right)\frac{Z_4^N(s)}{\sqrt{N}}\diff s.
\end{multline*}
It is easily checked that $(M_N(t))$ is a martingale whose previsible increasing process is given by, for $t{\ge}0$, 
\begin{multline*}
\croc{M}_N(t)=
\frac{\kappa_4}{\sqrt{N}}\int_0^{t} \ind{Z_3^N(s)=1}\\{\times}\left(\E\left.\left(\int_0^{E_{\kappa_0}} F\left(\frac{Z_4^N(s)-S(Z_4^N(s),u,(E_{\kappa_5,i}))}{\sqrt{N}}\right)^2\diff u\right|{\cal F}_s\right)\right)\frac{Z_4^N(s)}{\sqrt{N}}\diff s.
\end{multline*}
Using Corollary (\ref{CorZ}) one can show that there exists some finite constant $C_0$ such that
\[
\E\left(\croc{M}_N(t)\right)\leq \frac{C_0}{\sqrt{N}},
\]
the martingale $(M_N(t))$ is converging in distribution to $0$.

We are now investigating  the asymptotic behavior of $(J_N(t))$.

By using again Corollary~\ref{CorZ}, for $T{>}0$ and $\eps{>}0$, there exist constants $0{<}d_0{\le}D_0$ such that
\[
\limsup_{N\to+\infty}\E\left(\int_0^T \frac{Z_4^N(s)}{\sqrt{N}} \ind{Z_4^N(s){\not\in}[d_0\sqrt{N},D_0\sqrt{N}]}\diff s \right)\leq \eps,
\]
this is due to the fact that the limit of the occupation measure of $(Z_4^N(s)/\sqrt{N})$ is expressed with a  distribution $\Gamma_0$ of Definition~\ref{GamDef} and, in particular,  without a mass at $0$. 

Let $z{\ge}1$, then
\begin{multline}
  \E\left(\int_0^{E_{\kappa_0}} F\left(\frac{z-S(z,u,(E_{\kappa_5,i}))}{\sqrt{N}}\right)\diff u\right)
  \\=  \int_0^{+\infty}\E\left(F\left(\frac{z}{\sqrt{N}} \frac{1}{z}\sum_{i=1}^z \ind{E_{\kappa_5,i}{>}u}\right)\right)\P(E_{\kappa_0}{\ge}u)\diff u,
\end{multline}
for $\eta{>}0$ and $u{\ge}0$, the relation 
\[
\P\left( \left|\frac{1}{z}\sum_{i=1}^z \ind{E_{\kappa_5,i}{>}u}{-}e^{-\kappa_5 u}\right|\ge \eta\right)\leq\frac{1}{z\eta^2}
\]
holds. With the uniform continuity of $F$, we therefore obtain the relation
\[
\lim_{N\to+\infty}\sup_{z{\in}[d_0\sqrt{N},D_0\sqrt{N}]}\E\left(\left|F\left(\frac{z}{\sqrt{N}} \frac{1}{z}\sum_{i=1}^z \ind{E_{\kappa_5,i}{>}u}\right){-}F\left(\frac{z}{\sqrt{N}} e^{-\kappa_5 u}\right)\right|\right)=0. 
\]
We have therefore that, for the convergence in distribution, the sequence $(J_N(t))$ has the same asymptotic behavior as
\[
\left(\kappa_4\int_0^{t}\ind{Z_3^N(s{-})=1} \int_0^{+\infty} F\left(\frac{Z_4^N(s)}{\sqrt{N}}e^{-\kappa_5u}\right)e^{-\kappa_0u}\diff u \frac{Z_4^N(s)}{\sqrt{N}}\diff s\right).
\]
We can now use Corollary (\ref{CorZ}) and standard calculus to complete the proof of the theorem. 
\end{proof}

We now establish the  fact that the process  $(X_2^N(t))$ has indeed the same asymptotic behavior as $(U_2^N(t))$. 
\begin{proposition}\label{U=Z}
Let $(X_N(t))$ and $(U_N(t))$ be the solutions of the SDEs~\eqref{SDEm4} and~\eqref{SDEm42} with initial point $(0,y_N,1,v_N)$ and $(y_N,v_N)$ satisfies Relation~\eqref{InitIV}, then the two sequences
  \[
  \left(\frac{X_2^N(\sqrt{N}t)}{N}\right)\text{ and }
  \left(\frac{U_2^N(\sqrt{N}t)}{N}\right)
  \]
have the same limit for the convergence in distribution.
\end{proposition}
\begin{proof}
The proof follows the arguments used for the convergence in distribution of $(U_2^N(\sqrt{N}t)/N)$. Some adjustments are nevertheless necessary but the main ideas are essentially the same. We sketch the main lines of the proof. 

If  $(X_N(t))$ is a solution of SDEs~\eqref{SDEm4}, on the time interval $[0,\sqrt{N}T]$ the contribution of the Poisson process ${\cal P}_1$ to the coordinate $(X_2^N(t))$ is of the order of $\sqrt{N}$ which is negligible since  the order of magnitude considered for $(X_2^N(t))$  is $N$. Therefore, we can take this Poisson process out of the set of SDEs for $(X_N(t))$.

For $T{>}0$, it is not difficult to show that, with high probability, the values of the process $(X_2(\sqrt{N}t)/N)$ are in $(\delta,3/2)$ on the time interval $[0,T]$. When there is a new arrival for the chemical species $S_1$, it is transformed  into chemical species $S_3$ at rate at least $\kappa_2\delta N$. Hence, with high probability, on the time interval $[0,\sqrt{N}T]$, the values of the process $(X_1^N(t))$ are in the set $\{0,1\}$. 

A central argument for the convergence of $(U_2^N(\sqrt{N}t)/N)$ is Theorem~\ref{TheoCVZ}.  We define by $(\widetilde{Z}_N(t))$ the analogue of $(Z_N(t))$, i.e. the time-changed process $(X_N(t))$ with all time intervals where $X_3^N$ is null are removed. The process $(\widetilde{Z}_N(t))$ satisfies the analogue of the SDEs~\eqref{SDEZ} where the last SDE is replaced by
\begin{multline*}
\diff \widetilde{Z}_4^N(t)={\cal P}_3\left((0,\kappa_3\widetilde{Z}_2^N\widetilde{Z}_3^N(t{-})),\diff t\right){-}{\cal P}_{5}\left((0,\kappa_{5}\widetilde{Z}_4^N(t{-})),\diff t\right),\\
\qquad{-}\ind{\widetilde{Z}_3^N(t{-}){=}1}\displaystyle\int_{a,b}S\left(\widetilde{Z}_4^N(t{-}),a{+}\frac{c}{\kappa_2\widetilde{Z}_2^N(t{-})},b\right)\\
\times\widetilde{\cal N}\left(\left(\rule{0mm}{4mm}0,\kappa_4Z_4^N(t{-})\right],\diff t,\diff a,\diff b,\diff c\right),
\end{multline*}
where  $\widetilde{\cal N}$ be a  Poisson marked point process  on $\R_+^2{\times}\R_+{\times}\R_+^\N{\times}\R_+{\times}$ with intensity measure
\[
\diff s\otimes\diff t\otimes\kappa_0e^{{-}\kappa_0a}\diff a\otimes Q(\diff b)\otimes e^{{-}c}\diff c,
\]
where, as before,  $Q$ is the distribution of $(E_i)$ on $\R_+^\N$. The additional variable $c$ of the Poisson process $\widetilde{\cal N}$  in this SDE is due to the fact that when $\widetilde{Z}_1^N(t){=}1$ and $\widetilde{Z}_3^N(t){=}0$ , $\widetilde{Z}_1$ leaves the state $1$ at rate $\kappa_2 \widetilde{Z}_2^N(t{-})$. Because of the assumption on $(\widetilde{Z}_2^N(t{-}/N)$ in $(\delta,3/2)$,  a glance at the proof of Theorem~\ref{LimitOccApp} shows that, even with this extra term, the limit result of this theorem still holds with the same limits. The proof is then concluded as in the proof of Proposition~\ref{PropH}. 
\end{proof}

\printbibliography

\end{document}